\documentclass[12pt]{amsart}
\usepackage[margin=1in]{geometry}
\usepackage[english]{babel}
\usepackage[utf8]{inputenc}
\usepackage{subcaption}
\usepackage{amsmath}
\usepackage{amssymb}
\usepackage{amsfonts}
\usepackage{amsthm}
\usepackage{mathdots}
\usepackage{mathrsfs}
\usepackage[all]{xy}
\usepackage[pdftex]{graphicx}
\usepackage{color}
\usepackage{cite}
\usepackage{url}
\usepackage{indent first}
\usepackage[labelfont=bf,labelsep=period,justification=raggedright]{caption}
\usepackage[english]{babel}
\usepackage[utf8]{inputenc}
\usepackage{hyperref}
\usepackage[colorinlistoftodos]{todonotes}
\usepackage{tkz-fct}
\usepackage{tikz}
\usetikzlibrary{calc}
\usepackage{pgfplots}
\usepackage{multicol}
\PassOptionsToPackage{dvipsnames,svgnames}{xcolor}
\usepackage{textcomp}
\usepackage{bbm}
\usepackage[ruled,vlined]{algorithm2e}
\usepackage[shortlabels]{enumitem}
\usepackage{stackrel}
\DeclareMathOperator{\SL}{SL}

\newcommand{\eps}{\varepsilon}
\newcommand{\QED}{\hspace{\stretch{1}} $\blacksquare$}

\newcommand{\CC}{\mathbb{C}}

\newcommand{\NN}{\mathbb{N}}

\newcommand{\RR}{\mathbb{R}}
\newcommand{\ZZ}{\mathbb{Z}}

\allowdisplaybreaks

\theoremstyle{plain}
\newtheorem{thm}{Theorem}

\newtheorem{cor}[thm]{Corollaire}

\newtheorem{prop}[thm]{Proposition}

\theoremstyle{definition}

\theoremstyle{remark}
\newtheorem{rem}[thm]{Remark}

\numberwithin{equation}{section}
\numberwithin{thm}{section}

\begin{document}
        \title[Variations of Ramanujan's Identity for odd zeta values]{Dirichlet Series under standard convolutions: Variations on Ramanujan's Identity for odd zeta values}
        \author{Parth Chavan$^{\,1}$ Sarth Chavan$^{\,2}$ Christophe Vignat$^{\,3}$ and Tanay Wakhare$^{\,4}$}
        %arranged author list in alphabetical order according to surnames.
        \address{$^{1}$ Euler Circle, Palo Alto, California 94306.}
        \email{spc2005@outlook.com}
        \address{$^{2}$ Euler Circle, Palo Alto, California 94306.}
        \email{sarth5002@outlook.com}
        \address{$^{3}$ Department of Mathematics, Tulane University, New Orleans, Louisiana}
        \email{cvignat@tulane.edu}
        \address{$^{4}$ Department of Electrical Engineering and Computer Science, Massachusetts Institute Of Technology, Cambridge, Massachusetts.}
        \email{twakhare@mit.edu}
        \subjclass[2020]{Primary: $11$Mxx; Secondary: $11$B$68$ and $11$F$03$}
       \keywords{Riemann zeta function, Ramanujan's formula for $\zeta(2n+1)$, Dirichlet series.}  
\maketitle
\begin{abstract}
Inspired by a famous identity of Ramanujan, we propose a general formula linearizing the convolution of Dirichlet series as the sum of Dirichlet series with modified weights; its specialization produces new identities and recovers several identities derived earlier in the literature, such as the convolution of squares of Bernoulli numbers by A. Dixit and collaborators, or the convolution of Bernoulli numbers by Y. Komori and collaborators. 
\end{abstract}
\section{A Brief History and Introduction}

One of the famous identities given by Ramanujan that has attracted the attention of several mathematicians over the years is the following:
\begin{thm}[Ramanujan's formula for $\zeta(2n+1)$]\label{thm}
If $\alpha$ and $\beta$ are positive real numbers such that $\alpha\beta=\pi^{2}$
and if $n \in \ZZ\setminus{\{0\}}$, then we have 
\begin{align}
 \alpha^{-n}\left\{ \dfrac{1}{2}\,\zeta(2n+1)+\sum_{m=1}^{\infty}\dfrac{m^{-2n-1}}{e^{2\alpha m}-1}\right\} &-\left(-\beta\right)^{-n}\left\{ \dfrac{1}{2}\,\zeta(2n+1)+\sum_{m=1}^{\infty}\dfrac{m^{-2n-1}}{e^{2\beta m}-1}\right\}\nonumber 
\\&=2^{2n}\sum_{k=0}^{n+1}\dfrac{\left(-1\right)^{k-1}\mathcal{B}_{2k}\,\mathcal{B}_{2n-2k+2}}{\left(2k\right)!\left(2n-2k+2\right)!}\,\alpha^{n-k+1}\beta^{k}.\label{eq:Ramanujan_main}
\end{align}
where $\mathcal{B}_{n}$ denotes the $n$-th Bernoulli number and
$\zeta(s)$ represents the Riemann zeta function. 
\end{thm}
Theorem \ref{thm} appears as Entry 21 in Chapter 14 of Ramanujan's second
notebook \cite[173]{Notebooks 2}. It also appears in a formerly unpublished
manuscript of Ramanujan that was published in its original handwritten
form with his lost notebook \cite[formula (28), page. 318\textendash 322]{Lost Notebook}. For
a fascinating account of the history and an elementary proof of \textit{Ramanujan's
formula for $\zeta(2n+1)$} (henceforth simply \textit{Ramanujan's
identity}) we refer the reader to \cite{Chavan}, \cite{Berndt} and
\cite{Ramanujan}.

The first published proof of Theorem \ref{thm} is due to S.L. Marulkar \cite{SL} although he was not aware that this formula can be found in Ramanujan’s Notebooks.

 B. Berndt and A. Straub in their very recent article \cite{Berndt} show that
equation (\ref{eq:Ramanujan_main}) can be rewritten in terms of hyperbolic
cotangent sums and is equivalent to 
\[
\alpha^{-n}\sum_{m=1}^{\infty}\dfrac{\coth\left(\alpha m\right)}{m^{2n+1}}-\left(-\beta\right)^{-n}\sum_{m=1}^{\infty}\dfrac{\coth\left(\beta m\right)}{m^{2n+1}}=-2^{2n}\sum_{k=0}^{n+1}\dfrac{\left(-1\right)^{k-1}\mathcal{B}_{2k}\,\mathcal{B}_{2n-2k+2}}{\left(2k\right)!\left(2n-2k+2\right)!}\,\alpha^{n-k+1}\beta^{k}.\]
where as before $n$ is a positive integer and $\alpha,\beta\in\RR^{+}$
are such that $\alpha\beta=\pi^{2}$. 

%To the best of our knowledge,
%the first recorded proof of Ramanujan's formula for $\zeta(2n+1)$
%in terms of hyperbolic cotangent sums was given by T.S. %Nanjundiah
%\cite{Nanjundiah} in 1951. 
Later S. Kongsiriwong \cite{Gun} derived
several beautiful generalizations and analogues of this identity.
Several interesting corollaries can be derived from
 formula (\ref{eq:Ramanujan_main}) of Ramanujan. For instance if we substitute $\alpha=\beta=\pi$
in equation (\ref{eq:Ramanujan_main}), we deduce, for $n\geqslant0$ 
\begin{equation}
\zeta(2n+1)=\pi\left(2\pi\right)^{2n}\sum_{k=0}^{n+1}\dfrac{\left(-1\right)^{k-1}\mathcal{B}_{2k}\,\mathcal{B}_{2n-2k+2}}{\left(2k\right)!\left(2n-2k+2\right)!}-2\sum_{m=1}^{\infty}\dfrac{1}{m^{2n+1}\left(e^{2\pi m}-1\right)}
\label{eq1.1}
\end{equation}
a formula apparently due to Lerch \cite{Lerch}. Identity (\ref{eq1.1}) is quite remarkable since it tells us
that $\zeta(2n+1)$  is equal to a rational multiple
of $\pi^{2n+1}$  plus a  rapidly
convergent series.
%\\

 Over the years, many generalizations and analogues of Theorem \ref{eq:Ramanujan_main} of different
kinds were studied. We provide a new and very general viewpoint on this identity: our main result is Theorem \ref{main_theorem}, which provides a formula for the $n$-fold convolution of arbitrary Dirichlet series as a sum over $n$ Dirichlet series with modified weights. We then explore some special cases, and rederive some other generalizations of Ramanujan's formula. 

Our transformation is tailored towards producing quasimodular functions with a desired convolution as the corresponding error term. 

Let's now interpret Ramanujan's identity as a quasimodular transformation, by which we mean a transformation of the form $f\left(z\right)=f\left(-1/z\right)+\eps\left(z\right)$, where $\eps$ is some error function. This allows to interpret our main result as an elementary method to obtain modular-type transformations.
Consider the $n=-1$ case of \eqref{eq:Ramanujan_main}. After rearranging, we obtain that for $\alpha,\beta > 0$ such that  $\alpha \beta = \pi^2$,
$$ \sum_{m=1}^{\infty} \frac{1}{m\left(e^{2\alpha m}-1\right)} -  \sum_{m=1}^{\infty} \frac{1}{m\left(e^{2\beta m}-1\right)} = \frac{\beta-\alpha}{12}-\frac14\log\left(\frac{\alpha}{\beta}\right).$$
The series $\sum_{m=1}^{\infty} \frac{1}{m\left(e^{2\alpha m}-1\right)}$ is almost invariant under the transformation $\alpha \mapsto \left(\pi^2/\alpha\right)$, with a small error on the right-hand side. This identity is equivalent to
$$ \alpha^{\frac14}\,e^{-\frac{\alpha}{12}} \prod_{m=1}^\infty (1-e^{-2\alpha n}) = \beta^{\frac14}\,e^{-\frac{\beta}{12}} \prod_{m=1}^\infty \left(1-e^{-2\beta n}\right)$$
for $\alpha\beta = \pi^2$. This expresses the invariance of the Dedekind eta function $$\eta\left(\omega\right):= e^{\frac{2\pi \imath \omega}{24}}\prod_{m=1}^\infty\left(1 - e^{2\pi \imath n \omega}\right)
$$ under the modular change of variables $\omega \mapsto -\omega^{-1}$. 

In general, we can rewrite Ramanujan's series as (with $q = e^{-2\alpha}$)
$$\sum_{m=1}^{\infty}\frac{1}{m^{2n+1}\left(e^{2\alpha m}-1\right)} =\sum_{m=1}^{\infty}\frac{e^{-2\alpha m}}{m^{2n+1}\left(1-e^{-2\alpha m}\right)} = \sum_{m=1}^{\infty} \frac{q^m}{m^{2n-1}(1-q^m)}=\sum_{m=1}^{\infty} q^m \sum_{d \, | \, m} \frac{1}{d^{2n-1}}, $$
which features the complex divisor function (with $s\in \mathbb{C}$), $$\tau_s(n) := \sum_{d\,|\,n} d^s.$$
The right-hand side is known as an Eisenstein series, in this case over $\SL_2(\mathbb{Z})$.
%an Eisenstein series with a quasimodular transformation, which is a transformation $f(z) = f(-\frac{1}{z})+ \epsilon(z)$ for some error function $\epsilon(z)$, instead of the classic modular $f(z) = f(-\frac{1}{z})$. }

\iffalse
[ rewrite the Romik stuff
Now this is equal to $$-\left(\frac{G_{2n}(z)}{2\zeta(2n)}-1 \right) \frac{B_{2n}}{4n}$$
[ cite], with $G_{\ell}(z) := \sum_{m \geq 0}\sum_{n \geq 0} \frac{1}{(n+mz)^\ell}$ an Eisenstein series. Therefore, a transformation for Ramanujan's series is equivalent to a transformation for Eisenstein series.
Consider a compact convex set $K\subset \mathbb{R}^2$, symmetric around the $x$ and $y$ axes. Then define $h_K:\mathbb{R}\to \mathbb{R}$ as the function whose graph is the upper boundary of $K$, which then has compact support and is even. Then the $K$-summation of the doubly indexed sequence $a_{m,n}, m,n\geq 0$ is defined as
$$ \sum_{K} a_{m,n} := \lim_{\lambda\to 0} \sum_{\lambda K \cap \mathbb{Z}^2 }a_{m,n}.$$
Romik and Scherer then showed that the Eisenstein series $G_2$ satisfies an explicit $K$-summation based quasimodular transformation \cite[Theorem 5]{Romik}:
$$\sum_{K} \frac{1}{(n+mz)^2} - \sum_{m \geq 0}\sum_{n \geq 0} \frac{1}{(n+mz)^2} = \int_{0}^\infty \frac{h_K(t)}{z^2t^2-h_K(t)^2}.$$
\fi

Therefore, we can interpret Ramanujan's identity for odd zeta values as a quasimodular transformation for Eisenstein series, with the error given as a convolution of Bernoulli numbers. The first key reinterpretation is that the Bernoulli numbers in Ramanujan's identity are in fact zeta functions evaluated at integers. Then instead of considering Ramanujan's identity as a quasimodular transformation, we solely focus on the error term involving Bernoulli numbers, now interpreted as zeta functions. We show that Ramanujan's identity follows from linearizing the convolution of zeta functions. This new perspective gives a completely elementary method to prove quasimodular transformations. Under different specializations, this may allow us to discover new quasimodular forms with desired error term under the transformation $z\mapsto - \frac{1}{z}$.

Our main result shows that Ramanujan's formula is the special case of a general result for the convolution of arbitrary Dirichlet series, parametrized by a set of zeros $x_n$ and weights $a_n$. This follows our general program \cite{Christophe1, Christophe2} to show that various identities for zeta and multiple zeta functions are special cases of polynomial identities, which then reduce to zeta functions under the appropriate specialization. We proposed the definition of structural multiple zeta identities, which hold for the \textit{quasisymmetric zeta function} $\sum_{n=1}^\infty \frac{1}{x_n^s}$ rather than the standard Riemann zeta function $\sum_{n=1}^\infty \frac{1}{n^s}$. This is the analog of a zeta function in the ring of quasisymmetric functions. Given such a structural identity, we obtain many interesting special cases for free, such as identities for the zeta function built from the zeros of an entire function, finite polynomial identities, and linear combinations of zeta functions.

Here we continue this program, but we have to consider Dirichlet series with weights, since the weights also transform when we consider the convolution of Dirichlet series. The main advantage to our master theorem is that this reduces the complexity of the calculations to solely computing the zeta generating function $$\sum_{n=1}^{\infty}\zeta(n)\,z^n = \sum_{n=1}^{\infty} a_n \,\frac{z}{x_n-z},$$ as discussed in subsection \eqref{gf}.

\iffalse
In a series of recent papers, A. Dixit and collaborators []\cite{all of their papers} have derived a series of generalizations of Ramanujan's identity. Their method of proof is to rewrite the term corresponding to $\alpha$ as a Mellin transform over a kernel with the desired poles, then shifting the line of integration till we obtain the integral expression for the $\beta$ term. Bernoulli numbers then arise as the residues collected by shifting the contour of integration. For instance, for their generalization to squares of zeta functions, we start with the integral 
$$\frac{1}{2\pi i} \int_{c-i\infty}^{c+i\infty}\frac{\zeta(1-s)^2}{2\cos \left(\frac{\pi s}{2}\right)}x^{-s}dx$$
[ add in more details. Let's examine the poles in the above explicitly, and show how $B_n^2$ arises, and show that's its a rescaling of $\Omega$]
Our method of proof works in different parameter regimes: while we may obtain the same identity, the parameter values under which it holds may be different. An advantage of our method is that each result follows in a straightforward way from our Dirichlet series transformation, which bypasses the technicalities associated with the contour integral proofs. The only analytic steps that need to be justified are swapping summations, which follows from the absolute convergence of $\psi$ in the cases we consider.

\fi

\section{Main Results and Notations}
\subsection{Introduction}

The starting point of the present study is the realization that Ramanujan's
identity (\ref{eq:Ramanujan_main}) can be rephrased as a convolution
identity for the Riemann zeta function. More precisely, after replacing
the Bernoulli numbers by their zeta counterparts given by Euler's
formula for $\zeta(2n)$, namely
\[
\dfrac{\mathcal{B}_{2n}}{(2n)!}=2\,\dfrac{(-1)^{n-1}\zeta(2n)}{(2\pi)^{2n}},\,\,\,n\in \NN.
\]
Ramanujan's identity expresses the convolution of Riemann zeta functions
as an extended zeta function, namely a Dirichlet series, as follows.
\begin{prop}
\label{Ramanujan_zeta} Let $\alpha,\beta$ be positive numbers such that $\alpha\beta=\pi^{2}$ and let $n$ be a non-negative integer. An equivalent form of Ramanujan's identity (\ref{eq:Ramanujan_main}) is 
\begin{align*}
&\alpha^{-n}\left\{ \dfrac{1}{2}\,\zeta(2n+1)+\sum_{m=1}^{\infty}\dfrac{m^{-2n-1}}{e^{2\alpha m}-1}-\dfrac{1}{2\alpha}\,\zeta(2n+2)\right\} \\
&-\left(-\beta\right)^{-n}\left\{ \dfrac{1}{2}\,\zeta(2n+1)+\sum_{m=1}^{\infty}\dfrac{m^{-2n-1}}{e^{2\beta m}-1}-\dfrac{1}{2\beta}\,\zeta(2n+2)\right\} \\
& =\dfrac{\left(-1\right)^{n}}{\pi^{2n+2}}\sum_{k=1}^{n}\zeta(2n)\,\zeta(2n-2k+2)\,\alpha^{n-k+1}\beta^{k}.
\end{align*}
\end{prop}
It is thus natural to ask whether the convolution of an arbitrary Dirichlet
series is still a Dirichlet series. And surprisingly, the answer turns out to be positive, and our main result is about how Dirichlet series transform by convolution. When Dirichlet series are multiplied, we obtain the coefficients of the product in terms of a {Dirichlet convolution}, as 
$$\sum_{n= 1}^{\infty} \frac{a_n}{n^s} \sum_{m= 1}^{\infty} \frac{b_m}{m^s} = \sum_{n= 1}^{\infty} \frac{1}{n^s} \sum_{d \mid n}a_d\,b_{n/d}.$$ The key here is that we consider a standard convolution of the form $\sum_{k=0}^n a_k \,b_{n-k}$ instead, though classically Dirichlet series don't support a standard convolution structure.

\subsection{Notations} In this article, we consider Dirichlet series: for a sequence of non-zero complex numbers $\left\{ x_{n}\right\} $ that we will call \textbf{zeros} (see Section \ref{Bessel}) and a sequence of associated complex \textbf{weights} $\left\{ a_{n}\right\} $,
the Dirichlet series is thus defined as 
\begin{equation}\label{notation1}
    \zeta_{x,a}\left(N\right)=\sum_{n=1}^{\infty}\frac{a_{n}}{x_{n}^{N}}.
\end{equation}

Note that we assume that the Dirichlet series is convergent for $N\geqslant 1.$ If the series diverges at $N=1$, but has a finite abscissa of convergence, we can obtain analogous results. We call  the function 
\begin{equation}
\label{gf}    
\psi_{x,a}\left(z\right)=\sum_{N=1}^{\infty}\zeta_{x,a}\left(N\right)z^{N}
\end{equation}
associated to $\zeta_{x,a}$ the \textbf{zeta generating function}. Using a geometric series, we can check that the generating function $\psi_{x,a}\left(z\right)$ can be  expressed in terms
of the weights $\left\{ a_{n}\right\} $ and zeros $\left\{ x_{n}\right\} $
as follows 

\begin{equation}\label{notation2}
\psi_{x,a}\left(z\right)=\sum_{n=1}^{\infty}a_{n}\,\dfrac{z}{x_{n}-z}.
\end{equation}

When there is no ambiguity, the associated sequence of weights $\left\{ a_{n}\right\} $
will be omitted in the notations, so that for two Dirichlet series
$\zeta_{x,a}$ and $\zeta_{y,b},$ we will write simply
\[
\psi_{x,a}\left(z\right)=\psi_{x}\left(z\right),\,\,\zeta_{x,a}\left(N\right)=\zeta_{x}\left(N\right)\,\,\,\text{and}\,\,\,
\psi_{y,b}\left(z\right)=\psi_{y}\left(z\right),\,\,\zeta_{y,b}\left(N\right)=\zeta_{y}\left(N\right).
\]
Finally, we introduce the following notation: the modified sequence
of weights $\left\{ a.\psi_{y,b}\right\} _{n\geqslant 1}$ is 
\[
\left(a.\psi_{y,b}\right)_{n}=a_{n}\psi_{y,b}\left(x_{n}\right)
\]
so that the corresponding Dirichlet series is
\[
\zeta_{x,a.\psi_{y,b}}\left(N+1\right)=\sum_{n=1}^{\infty}\frac{a_{n}\psi_{y,b}\left(x_{n}\right)}{x_{n}^{N+1}}.
\]
Similarly, we will denote as $\left(a.\psi_{y,b}.\psi_{z,c}\right)$
the sequence defined by 
\[
\left(a.\psi_{y,b}.\psi_{z,c}\right)_{n}=a_{n}\psi_{y,b}\left(x_{n}\right)\psi_{z,c}\left(x_{n}\right)
\]
with associated Dirichlet series
\[
\zeta_{x,a.\psi_{y,b}.\psi_{z,c}}\left(N+1\right)=\sum_{n=1}^{\infty}\frac{a_{n}\psi_{y,b}\left(x_{n}\right)\psi_{z,c}\left(x_{n}\right)}{x_{n}^{N+1}}.
\]
Finally the convolution of two Dirichlet series is defined as 
\[
\left(\zeta_{y,b}*\zeta_{x,a}\right)\left(N+1\right)=\sum_{k=1}^{N}\zeta_{y,b}\left(k\right)\zeta_{x,a}\left(N+1-k\right)
\]
and the $n-$fold convolution of $n$ Dirichlet series $\zeta_{x^{\left(1\right)},a^{\left(1\right)}},\dots,\zeta_{x^{\left(n\right)},a^{\left(n\right)}}$
as
\[
\left(\stackrel[i=1]{n}{*}\zeta_{x^{\left(i\right)},a^{\left(i\right)}}\right)\left(N+1\right)=\sum\zeta_{x^{\left(1\right)},a^{\left(1\right)}}\left(k_{1}\right)\dots\zeta_{x^{\left(n\right)},a^{\left(n\right)}}\left(k_{n}\right)
\]
where the sum is over the set of indices 
\[
\left\{ \left(k_{1},k_2,\ldots,k_{n}\right):1\leqslant k_{i}\leqslant N,\sum_{i=1}^{n}k_{i}=N+1\right\}.
\]
\subsection{Main Results}
Our main result is a formula that expresses the $n-$fold convolution
of a set of Dirichlet series as a sum of the same $n$ Dirichlet series
with modified weights.
\begin{thm}
\label{main_theorem} For a set of $n\geqslant 2$ Dirichlet series $\left\{ \zeta_{x^{\left(i\right)},a^{\left(i\right)}}\right\} {}_{1\leqslant i\leqslant n}$,
we have, evaluated at argument $N+1$ removed for clarity,
\[
%\left(
\stackrel[i=1]{n}{*}\zeta_{x^{\left(i\right)},a^{\left(i\right)}}%\right)
%\left(N+1\right)
=\sum_{i=1}^{n}\zeta_{x^{\left(i\right)},a^{\left(i\right)}.\prod_{1\leqslant k\ne i\leqslant n}\psi_{x^{\left(k\right)}}}
%\left(N+1\right).
\]
The special case $n=2$ reads
\begin{align}
%\left(
\zeta_{y,b}*\zeta_{x,a}%\right)
%\left(N+1\right)
=\zeta_{x,a.\psi_{y}}%\left(N+1\right)
+\zeta_{y,b.\psi_{x}}%\left(N+1\right)
\label{eq:2terms}
\end{align}
while the case  $n=3$ is
\begin{align}
%\left(
\zeta_{z,c}*\zeta_{y,b}*\zeta_{x,a}%
%\right)%\left(N+1\right)
=\zeta_{x,a.\psi_{y}.\psi_{z}}%\left(N+1\right)
+\zeta_{y,b.\psi_{x}.\psi_{z}}%\left(N+1\right)
+\zeta_{z,c.\psi_{x}.\psi_{y}}%\left(N+1\right)
.\label{eq:3terms}
\end{align}
\end{thm}
For the sake of clarity, let us rephrase identities (\ref{eq:2terms}) and (\ref{eq:3terms}) in a more explicit way:
\[
\left(\zeta_{y,b}*\zeta_{x,a}\right)\left(N+1\right)=\sum_{n=1}^{\infty}\left\{\frac{a_{n}\psi_{y}\left(x_{n}\right)}{x_{n}^{N+1}}+\frac{b_{n}\psi_{x}\left(y_{n}\right)}{y_{n}^{N+1}}\right\},
\]
\[
\left(\zeta_{z,c}*\zeta_{y,b}*\zeta_{x,a}\right)\left(N+1\right)=\sum_{n=1}^{\infty}\left\{\frac{a_{n}\psi_{y}\left(x_{n}\right)\psi_{z}\left(x_{n}\right)}{x_{n}^{N+1}}+\frac{b_{n}\psi_{x}\left(y_{n}\right)\psi_{z}\left(y_{n}\right)}{y_{n}^{N+1}}+\frac{c_{n}\psi_{x}\left(z_{n}\right)\psi_{y}\left(z_{n}\right)}{z_{n}^{N+1}}\right\}.
\]
The proof of this Theorem \ref{main_theorem} is provided in Section
\ref{sec:proofs}. It also reveals a natural algebra underlying the convolution
of Dirichlet series that we now make explicit.
\\

Identity (\ref{eq:2terms}) is a simple consequence of a geometric
sum. Trying to deduce the three-fold convolution (\ref{eq:3terms})
from its two-fold counterpart (\ref{eq:2terms}), we use the associativity
of convolution to obtain
\begin{align*}
%\left(
\zeta_{z,c}*\zeta_{y,b}*\zeta_{x,a}%\right)%\left(N+1\right) 
& =%\left(
\zeta_{z,c}*\left(\zeta_{y,b}*\zeta_{x,a}\right)%\right)%\left(N+1\right)
\end{align*}
and now its distributivity with respect to the addition to produce
\[
%\left(
\zeta_{z,c}*\zeta_{y,b}*\zeta_{x,a}%\right)\left(N+1\right)
=%\left(
\zeta_{z,c}*\zeta_{x,a.\psi_{y}}%\right)\left(N+1\right)
+%\left(
\zeta_{z,c}*\zeta_{y,b.\psi_{x}}
%\right)\left(N+1\right).
\]
Expanding both terms according to rule (\ref{eq:2terms}) produces
\[
%\left(
\zeta_{z,c}*\zeta_{x,a.\psi_{y}}
%\right)\left(N+1\right)
=\zeta_{z,c.\psi_{x.a\psi_{y}}}%\left(N+1\right)
+\zeta_{x,a.\psi_{y}.\psi_{z,c}}%\left(N+1\right)
,
\]
and 
\[
%\left(
\zeta_{z,c}*\zeta_{y,b.\psi_{x}}
%\right)\left(N+1\right)
=\zeta_{z,c.\psi_{y,b.\psi_{x}}}
%\left(N+1\right)
+\zeta_{y,b.\psi_{x}.\psi_{z,c}}
%\left(N+1\right)
.
\]
Both Dirichlet series in $x$ and $y$ can be expressed in the more
simple way
\[
\zeta_{x,a.\psi_{y}.\psi_{z,c}}
%\left(N+1\right)
=\zeta_{x,a.\psi_{y}.\psi_{z}}
%\left(N+1\right) 
\,\,\,\text{and}\,\,\,
\zeta_{y,b.\psi_{x}.\psi_{z,c}}
%\left(N+1\right)
=\zeta_{y,b.\psi_{y}.\psi_{z}}%\left(N+1\right)
.
\]
The fact that sum of the two Dirichlet series in $z$ simplifies to
\[
\zeta_{z,c.\psi_{x.a\psi_{y}}}%\left(N+1\right)
+\zeta_{z,c.\psi_{y,b.\psi_{x}}}%\left(N+1\right)
=\zeta_{z,c.\psi_{x}.\psi_{y,}}%\left(N+1\right)
\]
is deduced from Lemma \ref{lem:Lemma} provided and proved in the
Section \ref{sec:proofs}, that we restate here 
\begin{equation}
\psi_{x.a\psi_{y}}\left(z\right)+\psi_{y,b.\psi_{x}}\left(z\right)=\psi_{x,a}\left(z\right)\psi_{y,b}\left(z\right),\label{eq:Lemma_identity}
\end{equation}
finally producing 
\[
%\left(
\zeta_{z,c}*\zeta_{y,b}*\zeta_{x,a}
%\right)\left(N+1\right)
=\zeta_{x,a.\psi_{y}.\psi_{z}}
%\left(N+1\right)
+\zeta_{y,b.\psi_{y}.\psi_{z}}
%\left(N+1\right)
+\zeta_{z,c.\psi_{x}.\psi_{y,}}
%\left(N+1\right)
\]
as expected. This completes the proof of equation (\ref{eq:3terms}). Let us notice that the \emph{sum to product}
identity for generating functions (\ref{eq:Lemma_identity}) that
underlies this convolution algebra for Dirichlet series is based on
the innocent looking \emph{sum to product}
identity
\[
\left(\frac{z}{y-z}\right)\left(\frac{y}{x-y}\right)+\left(\frac{z}{x-z}\right)\left(\frac{x}{y-x}\right)=\left(\frac{z}{y-z}\right)\left(\frac{z}{x-z}\right).
\]
\section{Particular cases and Extensions\label{sec:Particular-cases}}
This section shows how our main result allows to recover different
versions of convolution identities for Dirichlet functions, but also
how it can generate new ones. As an example, we derive new convolution
identities for the Bessel zeta and Hurwitz zeta functions. 
\subsection{Generalized Ramanujan}
\subsubsection{A first generalization}
Let us first consider the case $n=2$ in the Ramanujan setup.
\begin{thm}
\label{thm:Generalized Ramanujan} With arbitrary $\alpha,\beta>0,$
the choice %$x_n = \left(n^{2}\pi^{2}/\beta\right)$ and $y_n = -\left(n^{2}\pi^{2}/\alpha\right)$
\[
x_{n}=\dfrac{n^{2}\pi^{2}}{\beta},\,\,y_{n}=-\dfrac{n^{2}\pi^{2}}{\alpha}
\]
in equation $(\ref{eq:2terms})$ produces
%\begin{align}
\begin{equation}\label{eq:Ramanujan extended}\sum_{k=1}^{N}\left(-\beta\right)^{k}\alpha^{N+1-k}\,\zeta\left(2k\right)\zeta\left(2N+2-2k\right)  =\frac{\alpha^{N+1}+\left(-\beta\right)^{N+1}}{2}\,\zeta\left(2N+2\right) \end{equation}
\[-\frac{\alpha^{N+1}\pi}{2}\sum_{n=1}^{\infty}\frac{1}{n^{2N+1}}\sqrt{\frac{\beta}{\alpha}}\coth\left(\pi n\sqrt{\frac{\beta}{\alpha}}\right)  -\frac{\pi \left(-\beta\right)^{N+1}}{2}\sum_{n=1}^{\infty}\frac{1}{n^{2N+1}}\sqrt{\frac{\alpha}{\beta}}\coth\left(\pi n\sqrt{\frac{\alpha}{\beta}}\right).\]
%\end{align}
\end{thm}
Although it seems that we don't need the constraint $\alpha\beta=\pi^{2}$
as in equation (\ref{eq:Ramanujan_main}), the fact that identity (\ref{eq:Ramanujan extended}), now expressed in terms of the ratio $\mu=\beta/\alpha$,
%\begin{align*}
\[\sum_{k=1}^{N}\left(-1\right)^{k-1}\mu^{k}\,\zeta\left(2k\right)\zeta\left(2N+2-2k\right)  =-\frac{\mu^{N+1}+\left(-1\right)^{N+1}}{2}\,\zeta\left(2N+2\right)\]
\[+\,\frac{\pi}{2}\sqrt{\mu}\sum_{n=1}^{\infty}\frac{\coth\left(\pi n\sqrt{\mu}\right)}{n^{2N+1}}  +\frac{\pi\left(-1\right)^{N+1}}{2}\,\mu^{N+1/2}\sum_{n=1}^{\infty}\frac{1}{n^{2N+1}}\coth\left(\dfrac{\pi n}{\sqrt{\mu}}\right)\]
%\end{align*}
depends only on parameter $\mu,$ shows that it is a one
parameter identity equivalent to (\ref{eq:Ramanujan_main}).
\subsubsection{Bernoulli numbers}
We now make the following choice in equation (\ref{eq:2terms}):
\[
a_{m}=\exp\left({2\imath\pi my_{1}}\right),\,x_{m}=\frac{2\imath\pi m}{\omega_{1}},\,b_{m}=\exp\left({2\imath\pi my_{2}}\right),\,y_{m}=\frac{2\imath\pi m}{\omega_{2}}
\]
with $m\in \ZZ\setminus{\{0\}}$, so that we have
\[
\zeta_{x}\left(n\right)=-\omega_{1}^{n}\,\dfrac{\mathcal{B}_{n}\left(y_{1}\right)}{n!}\,\,\,\text{and}\,\,\,\zeta_{y}\left(n\right)=-\omega_{2}^{n}\,\dfrac{\mathcal{B}_{n}\left(y_{2}\right)}{n!}\quad \left(n \in \NN\right)
\]
This produces, assuming $0 \leqslant y_1 \leqslant 1$, $0 \leqslant y_2 \leqslant 1$ and $\Im \left(\omega_1/\omega_2\right)\ne 0,$  the identity
\begin{align*}
\tag{3.2}\label{eq:Bernoulli 2 terms}\sum_{k=0}^{N+1}\omega_{1}^{k}\,\frac{\mathcal{B}_{k}\left(y_{1}\right)}{k!}\,\omega_{2}^{N+1-k}\,\frac{\mathcal{B}_{N+1-k}\left(y_{2}\right)}{\left(N+1-k\right)!}  =&-\omega_{1}\sum_{m\in\ZZ\setminus{\{0\}}}\frac{e^{2\imath\pi ny_{2}}}{\left(\dfrac{2\imath\pi m}{\omega_{2}}\right)^{N}}\frac{e^{2\imath\pi m\frac{\omega_{1}}{\omega_{2}}y_{1}}}{e^{2\imath\pi m\frac{\omega_{1}}{\omega_{2}}}-1} \\ & -\omega_{2}\sum_{m\in \ZZ\setminus{\{0\}}}\frac{e^{2\imath\pi my_{1}}}{\left(\dfrac{2\imath\pi m}{\omega_{1}}\right)^{N}}\frac{e^{2\imath\pi m\frac{\omega_{2}}{\omega_{1}}y_{2}}}{e^{2\imath\pi m\frac{\omega_{2}}{\omega_{1}}}-1}.
\end{align*}
The proof of this identity is given in Section \ref{sec:proofs} and reveals the following
generalization.
\begin{thm}
\label{thm:Bernoulli n terms}If the coefficients $\left\{ \omega_{i}\right\} _{1\leqslant i\leqslant n}$ are such that $\Im \left(\omega_i/\omega_j\right)\ne 0,$ $ i\ne j,$ and if $0 \leqslant y_i \leqslant 1$ for  $1 \leqslant i \leqslant n$, then we have
\begin{align}
\sum_{k_{1},\ldots,k_{n}}\prod_{i=1}^{n}\omega_{i}^{k_{i}-1}\,\dfrac{\mathcal{B}_{k_i}\left(y_{i}\right)}{k_{i}!} & =-\sum_{i=1}^{n}\frac{1}{\omega_{i}}\sum_{m\ne0}\dfrac{e^{2\imath\pi my_{i}}}{\left(\dfrac{2\imath\pi m}{\omega_{i}}\right)^{N}}\prod_{j\ne i}\dfrac{e^{2\imath\pi m\frac{\omega_{j}}{\omega_{i}}y_{j}}}{e^{2\imath\pi m\frac{\omega_{j}}{\omega_{i}}}-1}.\label{eq:Bernoulli n terms}
\end{align}
where the sum on the left-hand side is over the set of indices
\[
\left\{ \left(k_{1},k_2,\ldots,k_{n}\right): 0\leqslant k_{i}\leqslant N+1,\sum_{i=1}^{n}k_{i}=N+1\right\}.
\]
\end{thm}
This identity is provided in \cite{Komori} in the special case $y=y_{1}=\cdots=y_{n}$
and is derived using properties of the Barnes zeta function. Our proof
relies solely on our main convolution result for Dirichlet zeta functions
and on the classical Fourier series expansions 
\[
\dfrac{1}{z^{2}}+2\sum_{m=1}^{\infty}\dfrac{1}{z^{2}+\frac{4\pi^{2}}{\omega_{1}^{2}}m^{2}}\cos\left(2\pi m y_{1}\right)=\dfrac{\omega_{1}}{2z}\dfrac{\cosh\left[\left(\omega_{1}z\right)\left(\frac{1}{2}-y_{1}\right)\right]}{\sinh\left(\dfrac{\omega_{1}z}{2}\right)}
\]
and
\[
\sum_{m=1}^{\infty}\frac{m}{z^{2}+\frac{4\pi^{2}}{\omega_{1}^{2}}m^{2}}\sin\left(2\pi m y_{1}\right)=\frac{\omega_{1}^{2}}{8\pi}\frac{\sinh\left[\left(\omega_{1}z\right)\left(\frac{1}{2}-y_{1}\right)\right]}{\sinh\left(\dfrac{\omega_{1}z}{2}\right)}
\]
that can be found for example as \cite[equations (1.51) and (1.53)]{Oberhettinger}.

A version of Theorem \ref{thm:Bernoulli n terms} for  Euler polynomials $E_{n}\left(x\right),$ defined by  generating function
\[\sum_{n = 0}^{\infty} \dfrac{E_{n}\left(x\right)}{n!}\,z^n = \dfrac{2}{e^x+1}\,e^{zx}
\] 
is  deduced next.
\begin{cor}
\label{cor:Euler}
If  $\left\{ \omega_{i}\right\} _{1\leqslant i\leqslant n}$
are such that $\Im\left(\omega_i/\omega_j\right) \ne 0,\thinspace\thinspace i\ne j,$ and if $0 \leqslant x_i \leqslant \frac{1}{2}$ for  $1 \leqslant i \leqslant n$, then we have
\[
\sum_{k_{1},k_2,\ldots,k_{n}}\prod_{i=1}^{n}\omega_{i}^{k_{i}-1}\,\frac{E_{k_{i}}\left(2x_{i}\right)}{k_{i}!}=\frac{2^{2N+n+3}}{\omega_{1}\omega_{2}\ldots\omega_{n}}\sum_{i=1}^{n}\sum_{m\,\textrm{odd}}\frac{\frac{e^{2\imath\pi mx_{i}}}{2\imath\pi m}}{\left(\dfrac{2\imath\pi m}{\omega_{i}}\right)^{N+n-1}}\prod_{j\ne i}\frac{e^{2\imath\pi m\frac{\omega_{j}}{\omega_{i}}x_{j}}}{e^{\imath\pi m\frac{\omega_{j}}{\omega_{i}}}+1}
\]
where as before the sum on the left-hand side is over the set of indices
\[
\left\{ \left(k_{1},k_2,\dots,k_{n}\right): 0\leqslant k_{i}\leqslant N+1,\sum_{i=1}^{n}k_{i}=N+1\right\}.
\]
\end{cor}

\subsubsection{Squares of Bernoulli numbers}

Our general Theorem \ref{main_theorem} deals with simple sums. However,
a quick look at its proof shows that they can be replaced by multiple
sums. We produce here the case of double sums. Starting from a sequence
of weights $\left\{ a_{m,n}\right\} $ and a sequence of roots $\left\{ x_{m,n}\right\} ,$ let us consider the double Dirichlet sum
\[
\zeta_{x,a}\left(N\right)=\sum_{m=1}^{\infty}\sum_{n=1}^{\infty}\frac{a_{m,n}}{x_{m,n}^{N}}
\]
and its associated generating function 
\[
\psi_{x,a}\left(z\right)=\sum_{N=1}^{\infty}\zeta_{x,a}\left(N\right)z^{N}=\sum_{m=1}^{\infty}\sum_{n=1}^{\infty}a_{m,n}\,\frac{z}{x_{m,n}-z}.
\]
\begin{thm}
\label{thm:twoindicesgeneral}For two double sequences of weights
$\left\{ a_{m,n}\right\}$ and $\left\{ b_{m,n}\right\} $ and two
sequences of zeros $\left\{ x_{m,n}\right\} $ and $\left\{ y_{m,n}\right\} $,
the convolution of the Dirichlet series $\zeta_{x,a}$ and $\zeta_{y,b}$
satisfies the identity
\[
\left(\zeta_{y,b}*\zeta_{x,a}\right)\left(N+1\right)=\sum_{m=1}^{\infty}\sum_{n=1}^{\infty}\left[\frac{a_{m,n}\psi_{y}\left(x_{m,n}\right)}{x_{m,n}^{N+1}}+\frac{b_{m,n}\psi_{x}\left(y_{m,n}\right)}{y_{m,n}^{N+1}}\right].\]
\end{thm}
The proof of this theorem is omitted here as it only requires replacing
simple sums by double sums in the proof of equation (\ref{eq:2terms}).

As a consequence of this result, we consider the simple choice
\[
a_{p,q}=1,\,\,b_{r,s}=1,\,\, x_{p,q}=-\frac{p^{2}q^{2}}{\alpha^{2}}\,\,\, \text{and}\,\,\,y_{r,s}=\frac{r^{2}s^{2}}{\beta^{2}}.\]
This produces the Dirichlet series
\[
\zeta_{x,a}\left(s\right)=\left(-\alpha^{2}\right)^{s}\zeta^{2}\left(2s\right),\,\,\zeta_{y,b}\left(s\right)=\beta^{2s}\,\zeta^{2}\left(2s\right)
\]
and the generating function
\[
\psi_{x}\left(z\right)=-\sum_{p=1}^{\infty}\sum_{q=1}^{\infty}\frac{\alpha^{2}z}{p^{2}q^{2}+\alpha^{2}z}=-\sum_{n=1}^{\infty}\tau_{0}(n)\,\frac{\alpha^{2}z}{n^{2}+\alpha^{2}z}
\]
with $\tau_{0}(n)$ as the number of divisors of $n$, and where
we have applied the general formula
\begin{align}
\sum_{p=1}^{\infty}\sum_{q=1}^{\infty}F(pq) & =\sum_{n=1}^{\infty}\sum_{q\mid n}F(n)=\sum_{n=1}^{\infty}\tau_{0}(n)\,F(n).\label{eq:d0formula}
\end{align}
Similarly, we find that
\[
\psi_{y}\left(z\right)=\sum_{m=1}^{\infty}\tau_{0}(m)\,\frac{\beta^{2}z}{m^{2}-\beta^{2}z}.
\]
Applying (\ref{eq:2terms})
%(\textcolor{red}{should it be 3.4 ? })\tw{(2.4) is right} 
produces the following identity.
\begin{cor}
\label{cor:zetasquare}For arbitrary $\alpha>0$ and $\beta>0,$ we
have 
%\begin{align*}
%&\sum_{k=1}^{N}\left(-1\right)^{k}\alpha^{2k}\beta^{2N+2-2k}\zeta^{2}\left(2k\right)\zeta^{2}\left(2N+2-2k\right) %\\=\left(-\alpha^{2}\right)^{N+1}&\sum_{p,q}\frac{1}{\left(pq\right)^{2N+2}}\sum_{m=1}^{\infty}\tau_{0}(m)\frac{\beta^{2}p^{2}q^{2}}{\alpha^{2}m^{2}+\beta^{2}p^{2}q^{2}}-\beta^{2N+2}\sum_{r,s}\frac{1}{\left(rs\right)^{2N+2}}\sum_{n=1}^{\infty}\tau_{0}(n)\frac{\alpha^{2}r^{2}s^{2}}{\beta^{2}n^{2}+\alpha^{2}r^{2}s^{2}}\\
%\\=-\left(-\alpha^{2}\right)^{N+1}&\sum_{m=1}^{\infty}\sum_{n=1}^{\infty}\frac{\tau_{0}(n)\,\tau_{0}(m)}{n^{2N}}\,\frac{\beta^{2}}{\alpha^{2}m^{2}+\beta^{2}n^{2}} \\ -\left(\beta^{2}\right)^{N+1}&\sum_{m=1}^{\infty}\sum_{n=1}^{\infty}\frac{\tau_{0}(n)\,\tau_{0}(m)}{m^{2N}}\,\frac{\alpha^{2}}{\beta^{2}n^{2}+\alpha^{2}m^{2}}.
%\end{align*}
\begin{align*}\sum_{k=1}^{N}\left(-1\right)^{k}\alpha^{2k}\beta^{2N+2-2k}\zeta^{2}(2k)\,\zeta^{2}(2N+2-2k)&= -\left(\beta^{2}\right)^{N+1}\sum_{m=1}^{\infty}\sum_{n=1}^{\infty}\frac{\tau_{0}(n)\,\tau_{0}(m)}{m^{2N}}\,\frac{\alpha^{2}}{\beta^{2}n^{2}+\alpha^{2}m^{2}}\\&\,\,-\left(-\alpha^{2}\right)^{N+1}\sum_{m=1}^{\infty}\sum_{n=1}^{\infty}\frac{\tau_{0}(n)\,\tau_{0}(m)}{n^{2N}}\,\frac{\beta^{2}}{\alpha^{2}m^{2}+\beta^{2}n^{2}}.\end{align*}
\end{cor}
This result provides an alternate but equivalent identity compared to the identity
by A. Dixit and R. Gupta \cite[Thm. 2.1]{DixitSquare}: assuming $\alpha,\beta >0$ and $\alpha \beta =\pi ^2$, 

%\begin{align*}
%\left(\alpha^{2}\right)^{-N}&\left\{ \zeta^{2}\left(2N+1\right)\left(\gamma+\log\left(\frac{\alpha}{\pi}\right)-\frac{\zeta'\left(2N+1\right)}{\zeta\left(2N+1\right)}\right)+\sum_{n=1}^{\infty}\frac{\tau_{0}\left(n\right)\Omega\left(\frac{\alpha^2}{\pi^2}n\right)}{n^{2N+1}}\right\} \\
%=\left(-\beta^{2}\right)^{-N}&\left\{ \zeta^{2}\left(2N+1\right)\left(\gamma+\log\left(\frac{\beta}{\pi}\right)-\frac{\zeta'\left(2N+1\right)}{\zeta\left(2N+1\right)}\right)+\sum_{n=1}^{\infty}\frac{\tau_{0}\left(n\right)\Omega\left(\frac{\beta^2}{\pi^2}n\right)}{n^{2N+1}}\right\} \\
%-\,2^{4N}\pi&\sum_{j=0}^{N+1}\frac{\left(-1\right)^{j}\mathcal{B}_{2j}^{2}\mathcal{B}_{2N+2-2j}^{2}}{\left(\left(2j\right)!\right)^{2}\left(\left(2N+2-2j\right)!\right)^{2}}\left(\alpha^{2}\right)^{j}\left(\beta^{2}\right)^{N+1-j}.
%\end{align*}
\begin{align*}
&\left(-\beta^{2}\right)^{-N}\left\{ \zeta^{2}\left(2N+1\right)\left(\gamma+\log\left(\frac{\beta}{\pi}\right)-\frac{\zeta'\left(2N+1\right)}{\zeta\left(2N+1\right)}\right)+\sum_{n=1}^{\infty}\frac{\tau_{0}(n)}{n^{2N+1}}\,\Omega\left(\dfrac{\beta^2 n}{\pi^2}\right)\right\}
\\ &  -\left(\alpha^{2}\right)^{-N}\left\{ \zeta^{2}\left(2N+1\right)\left(\gamma+\log\left(\frac{\alpha}{\pi}\right)-\frac{\zeta'\left(2N+1\right)}{\zeta\left(2N+1\right)}\right)+\sum_{n=1}^{\infty}\frac{\tau_{0}(n)}{n^{2N+1}}\,\,\Omega\left(\dfrac{\alpha^2 n}{\pi^2}\right)\right\}
\\ & = 2^{4N}\pi\sum_{j=0}^{N+1}\frac{\left(-1\right)^{j}\mathcal{B}_{2j}^{2}\mathcal{B}_{2N+2-2j}^{2}}{\left(\left(2j\right)!\right)^{2}\left(\left(2N+2-2j\right)!\right)^{2}}\left(\alpha^{2}\right)^{j}\left(\beta^{2}\right)^{N+1-j}.
\end{align*}
We have restated their transformation in terms of the Koshliakov kernel $\Omega(x)$, which has equivalent expressions \cite{DixitSquare}
\begin{align}
    \Omega(x)&= -\gamma - \frac12 \log x - \frac{1}{4\pi x} + \frac{x}{\pi} \sum_{j=1}^{\infty}\frac{\tau_0(j)}{x^2+j^2} \label{koshseries}\\
&=2\sum_{j=1}^{\infty}\tau_{0}(j)\left(K_{0}\left(4\pi\exp(\imath\pi/4)\sqrt{j x}\right)+K_{0}\left(4\pi \exp(-\imath\pi/4) \sqrt{j x}\right)\right).
\end{align} where $K_0$ represents the modified Bessel function of the second kind.

We note that although the definition of $\Omega$ is rather unmotivated, it is central to the theory of the \textit{Koshliakov zeta function} and naturally arises when generalizing Ramanujan's identity. Ramanujan's identity deals with the summation kernel $\left(\exp\left(2\pi x\right)-1\right)^{-1}$, %$\frac{1}{e^{2\pi x}-1}$, 
which has a pole with residue $1$ at $x=0$, and poles at $x=\pm \imath n, n\geqslant 1,$ with residues $\frac{1}{2\pi}$. The Koshliakov kernel also has a pole with residue $-\frac{1}{4\pi}$ at $x=0$, and poles at $x=\pm \imath n$ for $n\geqslant 1,$ with residues $\frac{\tau_0(n)}{2\pi}$ instead. This allows us to construct generalizations of Ramanujan's zeta identities involving divisor sums and Bessel functions instead.

%so that it's just a rescaling of the original Koshliakov function.

%
%To show the equivalence of these two transformations, the only tool we need is the partial fraction decomposition \eqref{koshseries}, where we take $x = \frac{\alpha}{\beta}m$. Let's consider the double sum:
%\begin{align*}
%& \sum_{m=1}^{\infty}\sum_{n=1}^{\infty}\frac{\tau_{0}(n)\,\tau_{0}(m)}{m^{2N}}\,\frac{\alpha^{2}}{\beta^{2}n^{2}+\alpha^{2}m^{2}} = \frac{\alpha \pi}{\beta}\sum_{m=1}^{\infty} \frac{\tau_0(m)}{m^{2N+1}} \left(\frac{\frac{\alpha}{\beta}m}{\pi}\sum_{n=1}^{\infty} \frac{\tau_0(n)}{n^2+ \left( \frac{\alpha}{\beta} m\right)^2}\right) \\
%&=\frac{\alpha \pi}{\beta}\sum_{m=1}^{\infty} \frac{\tau_0(m)}{m^{2N+1}} \left( \gamma+\frac12\log\left( \frac{\alpha}{\beta} m\right) + \frac{\beta}{4\pi m \alpha}+ \Omega\left(\frac{\alpha}{\beta} m\right)   \right) \\
%&= \frac{\alpha \pi}{\beta}\left( \gamma+\frac12 \log\left(\frac{\alpha}{\beta}\right) \right) \sum_{m=1}^{\infty} \frac{\tau_0(m)}{m^{2N+1}} +\frac{\alpha \pi}{2\beta}\sum_{m=1}^{\infty} \frac{\tau_0(m)\log m}{m^{2N+2}} + \frac14\sum_{m=1}^{\infty} \frac{\tau_0(m)}{m^{2N+2}} \\
%&\quad + \frac{\alpha \pi }{\beta}\sum_{m=1}^{\infty}\frac{\tau_0(m)\Omega( \frac{\alpha}{\beta}m )}{m^{2N+1}} \\
%&= \frac{\alpha \pi }{\beta}\left( \gamma+\frac12 \log\left(\frac{\alpha}{\beta}\right) \right) \zeta^2(2N+1) + \frac{\alpha \pi}{\beta }\zeta'(2N+2)\zeta(2N+2) +\frac{1 }{4}\zeta^2(2N+2) \\
%& \quad + \frac{\alpha \pi }{\beta}\sum_{m=1}^{\infty} \frac{\tau_0(m)\Omega( \frac{\alpha}{\beta}m )}{m^{2N+1}}
%\end{align*}
%

To show the equivalence of these two transformations, the only tool we need is the partial fraction decomposition \eqref{koshseries}. We write 
\begin{align*}
&\zeta^{2}\left(2N+1\right)\left(\gamma+\log\left(\frac{\alpha}{\pi}\right)-\frac{\zeta'\left(2N+1\right)}{\zeta\left(2N+1\right)}\right)+\sum_{n=1}^{\infty}\frac{\tau_{0}(n)}{n^{2N+1}}\,\Omega\left(\dfrac{\alpha^2 n}{\pi^2}\right)\\
&= \sum_{n=1}^{\infty} \frac{\tau_0(n)}{n^{2N+1}}\left(\gamma+\log\left(\frac{\alpha}{\pi}\right)+\frac12 \log n\right)+ \sum_{n=1}^{\infty}\frac{\tau_{0}(n)}{n^{2N+1}}\,\Omega\left(\dfrac{\alpha^2 n}{\pi^2}\right) \\
&= \sum_{n=1}^{\infty} \frac{\tau_0(n)}{n^{2N+1}}\left(\gamma+\log\left(\frac{\alpha}{\pi}\right)+\frac12 \log n\right) \\
& \quad+ \sum_{n=1}^{\infty} \frac{\tau_0(n)}{n^{2N+1}} \left(-\gamma - \log \left( \frac{\alpha\sqrt{n}}{\pi}\right)-\frac{1}{4\pi}\frac{\pi^2}{\alpha^2n}+\frac{\alpha^2}{\pi^2}\frac{n}{\pi} \sum_{j=1}^{\infty} \frac{\tau_0(j)}{j^2+\frac{\alpha^4}{\pi^4}n^2}\right) \\
&= \sum_{n=1}^{\infty} \frac{\tau_0(n)}{n^{2N+1}} \left(-\frac{\pi}{4\alpha^2 n}+\alpha^2 n \pi \sum_{j=1}^{\infty} \frac{\tau_0(j)}{\pi^4j^2+\alpha^4n^2} \right)\\
&= -\frac{\pi}{4\alpha^2}\,\zeta^2(2N+2) + \alpha^2\pi \sum_{n=1}^{\infty}\sum_{j=1}^{\infty} \frac{\tau_0(n)}{n^{2N}} \frac{\tau_0(j)}{\pi^4 j^2+\alpha^4 n^2}.
\end{align*}
We used the Dirichlet series identities $$\sum_{n=1}^{\infty} \frac{\tau_0(n)}{n^s} = \zeta^2(s)$$
and 
$$\sum_{m=1}^{\infty} \frac{\tau_0(m)\log m}{m^s} = \sum_{p=1}^{\infty}\sum_{q=1}^{\infty} \frac{\log pq}{p^sq^s} = 2 \sum_{p=1}^{\infty}\sum_{q=1}^{\infty} \frac{\log p}{p^sq^s} = -2\,\zeta(s)\,\zeta'(s).$$
Now repeat this argument with $\beta$, then add and use the fact that $\alpha\beta = \pi^2$. Finally, rewriting the zeta functions at even integers as Bernoulli numbers completes the proof.

With $c\in\mathbb{N},$ a natural extension of identity (\ref{eq:d0formula})
reads
\[
\sum_{p=1}^{\infty}\sum_{q=1}^{\infty}p^{c}F(pq)=\sum_{n=1}^{\infty}\tau_{c}(n)\,F(n). \]
so that, with the choice
\[a_{p,q}=q^{c},\,\,b_{r,s}=s^{d},\,\,x_{p,q}=-\frac{p^{2}q^{2}}{\alpha^{2}},\,\,y_{r,s}=\frac{r^{2}s^{2}}{\beta^{2}}\]
we obtain the Dirichlet functions 
\[\zeta_{x,a}\left(k\right)=\left(-1\right)^{k}\alpha^{2k}\sum_{p=1}^{\infty}\sum_{q=1}^{\infty}\frac{q^{c}}{p^{2k}q^{2k}}=\left(-1\right)^{k}\alpha^{2k}\zeta\left(2k\right)\zeta\left(2k-c\right)\]
and
\[
\zeta_{y,b}\left(k\right)=\beta^{2k}\sum_{r=1}^{\infty}\sum_{s=1}^{\infty}\frac{s^{d}}{r^{2k}s^{2k}}=\beta^{2k}\zeta\left(2k\right)\zeta\left(2k-d\right).
\]
This produces the following beautiful identity.
\begin{cor}\label{cor:shifted zeta}For arbitrary positive integers $c$ and $d,$ and
arbitrary 
% replaced positive by negative
negative real numbers $\alpha$ and $\beta,$ the following identity holds
%\begin{align*}
\begin{align*}&\sum_{k=1}^{N}\left(-1\right)^{k}\alpha^{2k}\beta^{2N+2-2k}\, \zeta\left(2k\right)\zeta\left(2k-c\right)\zeta\left(2N+2-2k\right)\zeta\left(2N+2-2k-d\right)  %\\ %=\sum_{p}\sum_{q}\frac{a_{p,q}}{x_{p,q}^{N+1}}\,\psi_{y}\left(x_{p,q}\right)+\sum_{r}\sum_{s}\frac{b_{r,s}}{y_{r,s}^{N+1}}\,\psi_{x}\left(y_{r,s}\right)\]
%=-\left(-\alpha^{2}\right)^{N+1}&\sum_{p=1}^{\infty}\sum_{q=1}^{\infty}\frac{q^{c}}{p^{2N+2}q^{2N+2}}\sum_{m=1}^{\infty}\tau_{d}(m)\,\frac{\beta^{2}p^{2}q^{2}}{\alpha^{2}m^{2}+\beta^{2}p^{2}q^{2}}\\
%  -\beta^{2N+2}&\sum_{r=1}^{\infty}\sum_{s=1}^{\infty}\frac{s^{d}}{r^{2N+2}s^{2N+2}}\sum_{n=1}^{\infty}\tau_{c}(n)\,\frac{\alpha^{2}r^{2}s^{2}}{\beta ^{2}n^{2}+\alpha^{2}r^{2}s^{2}}\\
\\&=-\left(-\alpha^{2}\right)^{N+1}\sum_{m=1}^{\infty}\sum_{n=1}^{\infty}\frac{\tau_{c}(n)\,\tau_{d}(m)}{n^{2N}}\frac{\beta^{2}}{\alpha^{2}m^{2}+\beta^{2}n^{2}}-\beta^{2N+2}\sum_{m=1}^{\infty}\sum_{n=1}^{\infty}\frac{\tau_{c}(n)\,\tau_{d}(m)}{m^{2N}}\frac{\alpha^{2}}{\beta^{2} n^{2}+\alpha^{2}m^{2}}.\end{align*}
%\end{align*}
\end{cor}
An open problem is to describe the right generalization of the Koshliakov kernel $\Omega(x)$, so that both free parameters $c$ and $d$ appear in this resulting transformation.

\subsection{The Bessel zeta case}
\label{Bessel}
When the numbers $\{x_n\}$ are the 
roots of an analytic function of order 1, so that the function possesses a Weierstrass infinite product representation 
\[
f\left( z \right) = f(0) \prod_{k = 1}^{\infty}
\left( 1-\frac{z}{x_k} \right),
\]
the generating function associated to the zeros $\{x_n\}$ and  with the coefficients $a_n=1$ is easily computed as
%\[
%\psi_x \left( z \right) = -z\frac{d}{dz}\log f(z)=-z\,\frac{f'\left( z\right)}{f\left( z\right)}.
%\]

\begin{align*}
-\frac{f'\left( z\right)}{f\left( z\right)} &= -z\,\frac{\mathrm{d}}{\mathrm{d}z}\log f(z) = -z\, \frac{\mathrm{d}}{\mathrm{d}z}\sum_{k=1}^{\infty} \log\left(1 - \frac{z}{x_k} \right) \\&= z\sum_{k=1}^{\infty} \frac{1/x_k}{1-z/x_k} = \sum_{k=1}^{\infty}\sum_{\ell=1}^{\infty} \left(\frac{z}{x_k}\right)^\ell = \sum_{\ell=1}^{\infty} \zeta_x(\ell)\,z^\ell.
\end{align*}
Bessel functions and their zeros produce an opportunity to test our
general formula in this case. One advantage of this parameterized family
of functions is that a special case for the value of the parameter,
namely $\nu=\frac{1}{2},$ recovers the previous Riemann zeta setup.
%Bessel zeta functions are also useful in quantum physics and number theory.
\\\\
The normalized Bessel function of the first kind $j_{\nu}\left(z\right)$,
defined as 
\begin{equation}\label{jnudef}
j_{\nu}\left(z\right)=2^{\nu}\,\Gamma\left(\nu+1\right)\frac{J_{\nu}\left(z\right)}{z^{\nu}},
\end{equation}
is an entire function such that $j_{\nu}\left(0\right)=1.$ It has
the infinite product expansion
\[
j_{\nu}\left(z\right)=\prod_{n = 1}^{\infty}\left(1-\frac{z^{2}}{j_{\nu,n}^{2}}\right)
\]
that reveals the real numbers $\left\{ j_{\nu,n}\right\} _{n\geqslant1}$
as the zeros of the Bessel function $J_{\nu}$. The Bessel zeta function
is defined in terms of these zeros as %should we renormalize so that this is $\zeta_{B,\nu}(2s)$ instead?}
\[
\zeta_{B,\nu}\left(s\right)=\sum_{n=1}^{\infty}\frac{1}{j_{\nu,n}^{2s}}
\]
and the corresponding zeta generating function reads
\[
\sum_{N=1}^{\infty} \zeta_{B,\nu}(N)\,z^N=-z\,\frac{\mathrm{d}}{\mathrm{d}z}\log j_{\nu}\left(\sqrt{z}\right)=\frac{z}{4\left(\nu+1\right)}\frac{j_{\nu+1}\left(\sqrt{z}\right)}{j_{\nu}\left(\sqrt{z}\right)}.
\] 
Therefore we deduce the following. 
\begin{thm}
\label{thm:The-Bessel-zeta} The Bessel zeta function satisfies the
convolution identity
\begin{align*}
&\sum_{k=1}^{N}\left(-1\right)^{k}\alpha^{N-k+1}\beta^{k}\,\zeta_{B,\nu}\left(k\right)\zeta_{B,\nu}\left(N+1-k\right)  =\nu\left(\alpha^{N+1}+\left(-\beta\right)^{N+1}\right)\zeta_{B,\nu}\left(N+1\right)\\
&-\frac{\alpha^{N+1}}{2}\sqrt{\frac{\beta}{\alpha}}\sum_{q=1}^{\infty}\frac{1}{j_{\nu,q}^{2N+1}}\frac{I_{\nu-1}\left(\sqrt{\dfrac{\beta}{\alpha}}\,j_{\nu,q}\right)}{I_{\nu}\left(\sqrt{\dfrac{\beta}{\alpha}}\,j_{\nu,q}\right)}  -\frac{\left(-\beta\right)^{N+1}}{2}\sqrt{\frac{\alpha}{\beta}}\sum_{q=1}^{\infty}\frac{1}{j_{\nu,q}^{2N+1}}\frac{I_{\nu-1}\left(\sqrt{\dfrac{\alpha}{\beta}}\,j_{\nu,q}\right)}{I_{\nu}\left(\sqrt{\dfrac{\alpha}{\beta}}\,j_{\nu,q}\right)}.
\end{align*}
\begin{cor}
The special case $\nu=\frac{1}{2}$ produces the zeros $j_{\frac{1}{2},n}=\pm n\pi,$ where $n\ne0$
so that
\[
j_{\frac{1}{2}}\left(z\right)=\frac{\sin\left(z\right)}{z}
\]
and
\[
\frac{I_{-\frac{1}{2}}\left(\sqrt{\dfrac{\beta}{\alpha}}\,q\pi\right)}{I_{\frac{1}{2}}\left(\sqrt{\dfrac{\beta}{\alpha}}\,q\pi\right)}=\coth\left(\sqrt{\frac{\beta}{\alpha}}\,q\pi\right),
\]
which recovers the Riemann zeta version of Ramanujan's identity.
\end{cor}
\end{thm}
The case $\nu=\frac{3}{2}$ is interesting since the ratio $\frac{I_{\frac{1}{2}}\left(z\right)}{I_{\frac{3}{2}}\left(z\right)}$
can be explicitly computed as
\[
\frac{I_{\frac{1}{2}}\left(z\right)}{I_{\frac{3}{2}}\left(z\right)}=\frac{z}{z\coth z-1}
\]
while the numbers $j_{\frac{3}{2},n}$ are the positive roots of the
equation
\[
\tan \left(x\right)=x.
\]
Hence we have the following corollary.
%\begin{cor}
\begin{cor}
Define $\left\{ z_{n}\right\} _{n\geqslant1}$ as the strictly positive
roots of the equation
\[
\tan\left(x\right)=x
\]
and the Bessel zeta function 
\[
\zeta_{B,\frac{3}{2}}\left(s\right)=\sum_{n=1}^{\infty}\frac{1}{z_{n}^{2s}},
\]
then the following identity holds
\begin{align*}
&\sum_{k=1}^{N}\left(-1\right)^{k-1}\alpha^{N-k+1}\beta^{k}\,\zeta_{B,\frac{3}{2}}\left(k\right)\zeta_{B,\frac{3}{2}}\left(N+1-k\right)  =-\frac{3}{2}\,\zeta_{B,\frac{3}{2}}\left(N+1\right)\left[\alpha^{N+1}+\left(-\beta\right)^{N+1}\right]+\\
&\alpha^{N}\beta\sum_{n=1}^{\infty}\frac{1}{z_{n}^{2N}}\frac{1}{\sqrt{\dfrac{\beta}{\alpha}}\,z_{n}\coth\left(\sqrt{\dfrac{\beta}{\alpha}}\,z_{n}\right)-1}+\left(-1\right)^{N+1}\beta^{N}\alpha\sum_{n=1}^{\infty}\frac{1}{z_{n}^{2N+1}}\frac{1}{\sqrt{\dfrac{\alpha}{\beta}}\,z_{n}\coth\left(\sqrt{\dfrac{\alpha}{\beta}}\,z_{n}\right)-1}.
\end{align*}
\end{cor}

\begin{cor}
The case $\nu=-\frac{1}{2}$ for which
\[
J_{-\frac{1}{2}}\left(z\right)=\sqrt{\frac{2}{\pi}}\,\frac{\cos z}{\sqrt{z}} \,\,\,\textrm{and}\,\,\,
j_{-\frac{3}{2},n}=\frac{\pi}{2}\left(2n-1\right),
\]
produces a zeta function that is the dissection of the Riemann zeta function,
\[
\zeta_{B,-\frac{1}{2}}\left(n\right)=\frac{2^{2n}-1}{\pi^{2n}}\,\zeta(2n).
\]
Since
\[
\frac{I_{-\frac{3}{2}}\left(z\right)}{I_{-\frac{1}{2}}\left(z\right)}=\tanh z-\frac{1}{z},
\]
we obtain the following identity after simplification
\begin{align*}
&\sum_{k=1}^{N}\left(-1\right)^{k}\alpha^{N-k+1}\beta^{k}\left(2^{2k}-1\right)\zeta\left(2k\right)\left(2^{2N+2-2k}-1\right)\zeta\left(2N+2-2k\right)\\
&=-\frac{1}{2}\left(\alpha^{N+1}+\left(-\beta\right)^{N+1}\right)\left(2^{2N+2}-1\right)\zeta\left(2N+2\right)\\
&\quad -\frac{\alpha^{N+1}}{2}\,\pi\,\sqrt{\frac{\beta}{\alpha}}\,\sum_{q=1}^{\infty}\frac{2^{2N}}{\left(2q-1\right)^{2N+1}}\tanh\left(\frac{\pi}{2}\sqrt{\frac{\beta}{\alpha}}\left(2q-1\right)\right)\\
&\quad+2^{2N}\left(\alpha^{N+1}\left(-\beta\right)^{N+1}\right)\frac{2^{2N+2}-1}{\pi^{2N+2}}\,\zeta\left(2N+2\right)\\
&\quad-\frac{\left(-\beta\right)^{N+1}}{2}\,\pi\,\sqrt{\frac{\alpha}{\beta}}\,\sum_{q=1}^{\infty}\frac{2^{2N}}{\left(2q-1\right)^{2N+1}}\tanh\left(\frac{\pi}{2}\sqrt{\frac{\alpha}{\beta}}\left(2q-1\right)\right)
\end{align*}
\end{cor}
%\[-\frac{1}{2}\left(\alpha^{N+1}+\left(-\beta\right)^{N+1}\right)\left(2^{2N+2}-1\right)\zeta\left(2N+2\right)-\frac{\alpha^{N+1}}{2}\sqrt{\frac{\beta}{\alpha}}\pi\sum_{q\ge1}\frac{2^{2N}}{\left(2q-1\right)^{2N+1}}\tanh\left(\sqrt{\frac{\beta}{\alpha}}\frac{\pi}{2}\left(2q-1\right)\right)\]
\subsection{The Hurwitz zeta case}
We look now at the Hurwitz zeta function defined by
\begin{equation}
\zeta_{H}\left(s;x\right):=\sum_{p=0}^{\infty}\frac{1}{\left(p+x\right)^{s}}\label{eq:Hurwitz}
\end{equation}
and choose
\[
x_{n}=\frac{\left(n+x\right)^{2}}{\beta},\,\,y_{n}=-\frac{\left(n+y\right)^{2}}{\alpha}
\]
in equation (\ref{eq:2terms}), which requires the generating function
\[
\sum_{n=0}^{\infty}\frac{z}{\left(n+x\right)^{2}-z}=\frac{\sqrt{z}}{2}\left[\psi\left(x+\sqrt{z}\right)-\psi\left(x-\sqrt{z}\right)\right],
\]
where, in this whole subsection, $\psi\left(x\right)$ denotes the digamma function (not to be confused with the zeta generating function \eqref{gf}).

After simplification, the resulting equivalent identity is as follows.
\begin{thm}
\label{thm:Hurwitz}The equivalent Ramanujan identity for the Hurwitz
zeta function (\ref{eq:Hurwitz}), with $\psi$ denoting the digamma function, is
\begin{align}
&\sum_{k=1}^{N}\left(-\alpha\right)^{N+1-k}\beta^{k}\zeta_{H}\left(2k;x\right)\zeta_{H}\left(2N+2-2k;y\right)\nonumber \\
&=\dfrac{\beta^{N+1}}{2}\,i\sqrt{\frac{\alpha}{\beta}}\sum_{n=0}^{\infty}\frac{1}{\left(x+n\right)^{2N+1}}\left[\psi\left(y+i\sqrt{\frac{\alpha}{\beta}}\left(x+n\right)\right)-\psi\left(y-i\sqrt{\frac{\alpha}{\beta}}\left(x+n\right)\right)\right]\\
&+\frac{\left(-\alpha\right)^{N+1}}{2}\,i\sqrt{\frac{\beta}{\alpha}}\sum_{n=0}^{\infty}\frac{1}{\left(y+n\right)^{2N+1}}\left[\psi\left(x+i\sqrt{\frac{\beta}{\alpha}}\left(y+n\right)\right)-\psi\left(x-i\sqrt{\frac{\beta}{\alpha}}\left(y+n\right)\right)\right]\nonumber\label{eq:Hurwitz_main}\end{align}
\end{thm}

\subsection{A Multisection case} In an attempt to specialize our main result (\ref{eq:2terms}), we look
for an identity that involves an usual odd $\zeta\left(2N+1\right)$
term together with a $\zeta\left(2mN+1\right)$ term for a positive
odd integer $m$ (this requirement will be explained in the forthcoming remark \ref{rem3.13}). This can be achieved using the following choice.
\begin{thm}
\label{Dixit_thm}For $m$ an odd integer and $\alpha,\beta>0$ arbitrary
real numbers, the choices 
\begin{equation}
a_{n}=1,\,\,x_{n}=\frac{n^{2}}{\alpha},\,\,b_{n}=n^{m-1} \,\,\,\text{and}\label{eq:choice1}\,\,\, y_{n}=-\frac{n^{2m}}{\beta}
\end{equation}
and the use of multisection identity 
\begin{equation}
\sum_{n=1}^{\infty}\frac{z^{2m}n^{m-1}}{z^{2m}-n^{2m}}=\frac{\pi z^{m}}{2m}\sum_{j=-\frac{m-1}{2}}^{\frac{m-1}{2}}\left(-1\right)^{j}\cot\left(\pi ze^{\imath\frac{j\pi}{m}}\right)\label{eq:multisection}
\end{equation}
produce the identity
\begin{align}
&\sum_{k=1}^{N}\alpha^{k}\left(-\beta\right)^{N+1-k}\zeta\left(2k\right)\zeta\left(2m\left(N-k\right)+m+1\right)  \nonumber \\
&=\alpha^{N+1}\frac{\pi}{2m}\sqrt{\frac{\beta}{\alpha}}\,\sum_{n=1}^{\infty}\frac{1}{n^{2N+1}}\sum_{j=-\frac{m-1}{2}}^{\frac{m-1}{2}}\left(-1\right)^{j}\coth\left(\pi n^{\frac{1}{m}}\left(\sqrt{\frac{\beta}{\alpha}}\right)^{\frac{1}{m}}e^{\imath\frac{j\pi}{m}}\right)\label{eq:multisection identity}\nonumber\\
&+\frac{\left(-\beta\right)^{N+1}}{2}\,\,\zeta\left(2mN+m+1\right)-\left(-\beta\right)^{N+1}\frac{\pi}{2}\sqrt{\frac{\alpha}{\beta}}\,\sum_{n=1}^{\infty}\frac{1}{n^{2mN+1}}\coth\left(\pi\sqrt{\frac{\alpha}{\beta}}\,n^{m}\right) .
\end{align}
\end{thm}
\begin{rem}
\label{rem3.13}
The requirement for the integer $m$ to be odd is due to the fact
that identity (\ref{eq:multisection}) holds in this case only. We
are unaware of an equivalent identity in the case where $m$ is even.
In \cite{Dixit}, A. Dixit and B. Maji derived the following identity:
assuming $\alpha\beta^{N}=\pi^{N+1},$ 
\begin{align}
&\left(\alpha^{\frac{2m}{m+1}}\right)^{-N}\left(\frac{1}{2}\,\zeta\left(2mN+1\right)+\sum_{n=1}^{\infty}\frac{1}{n^{2mN+1}\left(e^{\left(2n\right)^{m}\alpha}-1\right)}\right)\nonumber \\
&=\left(-\beta^{\frac{2m}{m+1}}\right)^{-N}\frac{2^{2N\left(m-1\right)}}{m}\left(\frac{\zeta(2N+1)}{2}+(-1)^{\frac{m+3}{2}}\sum_{j=-\frac{m-1}{2}}^{\frac{m-1}{2}}(-1)^{j}\sum_{n=1}^{\infty}\frac{1}{n^{2N+1}\left(e^{\left(2n\right)^{\frac{1}{m}}\beta\exp\left(\frac{\imath\pi j}{m}\right)}-1\right)}\right)\nonumber\\
&+\left(-1\right)^{N+\frac{m+1}{2}}2^{2mN}\sum_{j=0}^{\left\lfloor N+\frac{m+1}{2m}\right\rfloor }\left(-1\right)^{j}\frac{\mathcal{B}_{2j}}{\left(2j\right)!}\frac{\mathcal{B}_{2m\left(N-j\right)+m+1}}{\left(2m\left(N-j\right)+m+1\right)!}\,\alpha^{\frac{2j}{m+1}}\beta^{m+\frac{2m^{2}\left(N-j\right)}{m+1}}.
\label{eq:Atul}
\end{align}
This identity is a special case of identity (\ref{eq:multisection identity})
with the additional constraint $\alpha\beta^{N}=\pi^{N+1}.$
\end{rem}
%\newpage
%The choice $x_n = -\left(n^2/\beta\right), y_n = \left(n^2/\alpha\right)$ and  $a_n = b_n = 1/n$ in our general setup allows us to recover the following beautiful transformation for a combination of the Vlasenko--Zagier higher Herglotz function $F_k(x)$ as derived by A. Dixit and collaborators in \cite[Corollary 2.4]{DixitHerglotz}
\subsection{Higher Herglotz function analogue of Ramanujan's Identity} 
As a special case, Theorem \ref{main_theorem} gives the following %beautiful 
transformation for a combination of the Vlasenko--Zagier higher Herglotz function $F_k(x)$ which is analogous to Ramanujan's identity (\ref{eq:Ramanujan_main}) and is derived by A. Dixit and collaborators in \cite[Corollary 2.4]{DixitHerglotz}. 
\begin{prop}\label{Atul_Herglotz}
Let $\alpha$ and $\beta$ be two complex numbers such that $\Re(\alpha)>0, \Re(\beta)>0$ and $\alpha\beta=4\pi^2$. Let $\psi$ denote the digamma function. Then for $m \in \NN$, we have
\begin{align*}&\left(-\beta\right)^{-m}\left\{2\gamma\zeta(2m+1)+\sum_{n=1}^{\infty}\dfrac{1}{n^{2m+1}}\left(\psi\left(\dfrac{\imath n\beta}{2\pi}\right)+\psi\left(-\dfrac{\imath n\beta}{2\pi}\right)\right)\right\}\\&+\alpha^{-m}\left\{2\gamma\zeta(2m+1)+\sum_{n=1}^{\infty}\dfrac{1}{n^{2m+1}}\left(\psi\left(\dfrac{\imath n\alpha}{2\pi}\right)+\psi\left(-\dfrac{\imath n\alpha}{2\pi}\right)\right)\right\}\\&=-2\sum_{k=1}^{m-1}\left(-1\right)^{k}\zeta(2k+1)\,\zeta(2m-2k-1)\alpha^{k-m}\beta^{-k}.\end{align*}
\end{prop}
Letting $m = 1$ in Proposition $\ref{Atul_Herglotz}$ gives the following %beautiful
modular relation.
\begin{cor}
Let $\alpha,\beta\in \CC$ such that $\Re\left(\alpha\right)>0, \Re\left(\beta\right)>0$ and $\alpha\beta=4\pi^2$. Let $\psi$ denote the digamma function. Then, we have
\begin{align*}&\dfrac{1}{\alpha}\left\{2\gamma\zeta(3)+\sum_{n=1}^{\infty}\dfrac{1}{n^{3}}\left(\psi\left(\dfrac{\imath n\alpha}{2\pi}\right)+\psi\left(-\dfrac{\imath n\alpha}{2\pi}\right)\right)\right\}\\&=\dfrac{1}{\beta}\left\{2\gamma\zeta(3)+\sum_{n=1}^{\infty}\dfrac{1}{n^{3}}\left(\psi\left(\dfrac{\imath n\beta}{2\pi}\right)+\psi\left(-\dfrac{\imath n\beta}{2\pi}\right)\right)\right\}.\end{align*}
\end{cor}
%\newpage
\section{Proofs\label{sec:proofs}}
\subsection{Proof of Proposition \ref{Ramanujan_zeta}}
Ramanujan's formula for $\zeta(2n+1)$ (\ref{eq:Ramanujan_main}) 
states that 
\begin{align*}
\alpha^{-n}\left\{ \dfrac{1}{2}\,\zeta(2n+1)+\sum_{m=1}^{\infty}\dfrac{m^{-2n-1}}{e^{2\alpha m}-1}\right\} &-\left(-\beta\right)^{-n}\left\{ \dfrac{1}{2}\,\zeta(2n+1)+\sum_{m=1}^{\infty}\dfrac{m^{-2n-1}}{e^{2\beta m}-1}\right\} 
\\ &
=2^{2n}\sum_{k=0}^{n+1}\dfrac{\left(-1\right)^{k-1}\mathcal{B}_{2k}\,\mathcal{B}_{2n-2k+2}}{\left(2k\right)!\left(2n-2k+2\right)!}\,\alpha^{n-k+1}\beta^{k}.
\end{align*}
We start rewriting this identity by extracting the boundary terms
in the right-hand side sum 
\[\alpha^{-n}\left\{ \dfrac{1}{2}\,\zeta(2n+1)+\sum_{m=1}^{\infty}\dfrac{m^{-2n-1}}{e^{2\alpha m}-1}\right\} -\left(-\beta\right)^{-n}\left\{ \dfrac{1}{2}\,\zeta(2n+1)+\sum_{m=1}^{\infty}\dfrac{m^{-2n-1}}{e^{2\beta m}-1}\right\} \]
\[
=2^{2n}\sum_{k=0}^{n}\,\dfrac{(-1)^{k-1}\,\mathcal{B}_{2k}\,\mathcal{B}_{2n-2k+2}}{(2k)!\,(2n-2k+2)!}\,\,\alpha^{n-k+1}\beta^{k}-\dfrac{(-1)^{n}\,2^{2n}\,\mathcal{B}_{2n+2}}{(2n+2)!}\left(\alpha^{n+1}-\beta^{n+1}\right).
\]
Euler's formula for $\zeta(2n)$, namely 
\[
\dfrac{\mathcal{B}_{2n}}{\left(2n\right)!}=2\,\dfrac{\left(-1\right)^{n-1}\zeta(2n)}{\left(2\pi\right)^{2n}},
\]
 can then be applied to all terms in this sum, producing 
 \begin{align*}
&\alpha^{-n}\left\{ \dfrac{1}{2}\,\zeta(2n+1)+\sum_{m=1}^{\infty}\dfrac{m^{-2n-1}}{e^{2\alpha m}-1}\right\} -\left(-\beta\right)^{-n}\left\{ \dfrac{1}{2}\,\zeta(2n+1)+\sum_{m=1}^{\infty}\dfrac{m^{-2n-1}}{e^{2\beta m}-1}\right\} \\
& =\dfrac{(-1)^{n}\,2^{2n+2}}{(2\pi)^{2n+2}}\sum_{k=1}^{n}\zeta(2n)\,\zeta(2n-2k+2)\,\alpha^{n-k+1}\beta^{k}-\dfrac{(-1)^{2n}\,2^{2n+1}}{(2\pi)^{2n+2}}\,\zeta(2n+2)\left(\alpha^{n+1}-\beta^{n+1}\right)\\
&=\dfrac{(-1)^{n}\,2^{2n+2}}{(2\pi)^{2n+2}}\sum_{k=1}^{n}\zeta(2n)\,\zeta(2n-2k+2)\,\alpha^{n-k+1}\beta^{k}-\dfrac{1}{2}\,\zeta(2n+2)\,\dfrac{1}{\pi^{2n+2}}\left(\alpha^{n+1}-\beta^{n+1}\right)\\
&=\dfrac{(-1)^{n}}{\pi^{2n+2}}\sum_{k=1}^{n}\zeta(2n)\,\zeta(2n-2k+2)\,\alpha^{n-k+1}\beta^{k}-\dfrac{1}{2}\,\zeta(2n+2)\,\dfrac{1}{(\alpha\beta)^{n+1}}\left(\alpha^{n+1}-\beta^{n+1}\right)\\
&=\dfrac{(-1)^{n}}{\pi^{2n+2}}\sum_{k=1}^{n}\zeta(2n)\,\zeta(2n-2k+2)\,\alpha^{n-k+1}\beta^{k}+\dfrac{1}{2}\,\zeta(2n+2)\left(\dfrac{1}{\alpha^{n+1}}-\dfrac{1}{(-\beta)^{n+1}}\right)\\
    &=\dfrac{(-1)^{n}}{\pi^{2n+2}}\sum_{k=1}^{n}\zeta(2n)\,\zeta(2n-2k+2)\,\alpha^{n-k+1}\beta^{k}\,+\,\alpha^{-n}\left(\dfrac{1}{2\alpha}\,\zeta(2n+2)\right)-\,(-\beta)^{-n}\left(\dfrac{1}{2\beta}\,\zeta(2n+2)\right)
\end{align*}
which is the desired result. 
\subsection{Proof of Theorem \ref{main_theorem}}
The case with two terms ($n=2$) is a simple consequence of a geometric sum:
\begin{align*}
\sum_{k=1}^{N}\zeta_{x,a}\left(k\right)\zeta_{y,b}\left(N+1-k\right) & =\sum_{m=1}^{\infty}\sum_{n=1}^{\infty}\sum_{k=1}^{N}\frac{a_{n}}{x_{n}^{k}}\frac{b_{m}}{y_{m}^{N+1-k}}\\
 & =\sum_{m=1}^{\infty}\sum_{n=1}^{\infty}\frac{a_{n}b_{m}}{y_{m}^{N+1}}\sum_{k=1}^{N}\left(\frac{y_{m}}{x_{n}}\right)^{k}\\
 & =\sum_{m=1}^{\infty}\sum_{n=1}^{\infty}\frac{a_{n}b_{m}}{y_{m}^{N}}\frac{1}{x_{n}-y_{m}}\left[1-\left(\frac{y_{m}}{x_{n}}\right)^{N}\right]
 \\&=\sum_{m=1}^{\infty}\sum_{n=1}^{\infty}a_{n}b_{m}\left\{ \frac{1}{y_{m}^{N}}\frac{1}{x_{n}-y_{m}}-\frac{1}{x_{n}^{N}}\frac{1}{x_{n}-y_{m}}\right\} \\
    & =\sum_{m=1}^{\infty}\sum_{n=1}^{\infty}\frac{b_{m}}{y_{m}^{N+1}}\,\psi_{x,a}\left(y_{m}\right)+\sum_{n=1}^{\infty}\frac{a_{n}}{x_{n}^{N+1}}\,\psi_{y}\left(x_{n}\right)
 %\\
 %& =\sum_{m=1}^{\infty}\sum_{n=1}^{\infty}a_{n}b_{m}\left\{ \frac{1}{y_{m}^{N}}\frac{1}{x_{n}-y_{m}}-\frac{1}{x_{n}^{N}}\frac{1}{x_{n}-y_{m}}\right\} %\\
 %& =\sum_{m=1}^{\infty}\sum_{n=1}^{\infty}\frac{b_{m}}{y_{m}^{N+1}}\psi_{x,a}\left(y_{m}\right)+\sum_{n=1}^{\infty}\frac{a_{n}}{x_{n}^{N+1}}\psi_{y}\left(x_{n}\right).
\end{align*}
as desired. The case with $n=3$ terms is proved right after its statement in Theorem \ref{main_theorem},
and requires the following result. 
\begin{prop}
\label{lem:Lemma}The generating functions $\psi_{x,a}$ and $\psi_{y,b}$
of two arbitrary Dirichlet series $\zeta_{x,a}$ and $\zeta_{y,b}$
satisfy the composition rule 
\[
\psi_{x.a\psi_{y}}\left(z\right)+\psi_{y,b.\psi_{x}}\left(z\right)=\psi_{x,a}\left(z\right)\psi_{y,b}\left(z\right).
\]
\end{prop}
\begin{proof}
We have 
\[
\psi_{x,a\psi_{y}}\left(z\right)+\psi_{y,b.\psi_{x}}\left(z\right)=\sum_{p=1}^{\infty}\sum_{q=1}^{\infty}a_{p}\,b_{q}\left[\left(\frac{z}{y_{q}-z}\right)\left(\frac{y_{q}}{x_{p}-y_{q}}\right)+\left(\frac{z}{x_{p}-z}\right)\left(\frac{x_{p}}{y_{q}-x_{p}}\right)\right]
\]
and since 
\[
\left(\frac{z}{y_{q}-z}\right)\left(\frac{y_{q}}{x_{p}-y_{q}}\right)+\left(\frac{z}{x_{p}-z}\right)\left(\frac{x_{p}}{y_{q}-x_{p}}\right)=\left(\frac{z}{y_{q}-z}\right)\left(\frac{z}{x_{p}-z}\right),
\]
we deduce
that\begin{align*}
\psi_{x,a\psi_{y}}\left(z\right)+\psi_{y,b.\psi_{x}}\left(z\right) & =\sum_{p=1}^{\infty}\sum_{q=1}^{\infty}a_{p}\,b_{q}\left(\frac{z}{y_{q}-z}\right)\left(\frac{z}{x_{p}-z}\right) \\&=\sum_{p=1}^{\infty}a_{p}\left(\frac{z}{x_{p}-z}\right)\sum_{q=1}^{\infty}b_{q}\left(\frac{z}{y_{q}-z}\right)
 \\&=\psi_{x,a}\left(z\right)\psi_{y,b}\left(z\right).
\end{align*}
which is the desired result. %This completes the proof of Lemma \ref{lem:Lemma}. 
\end{proof}
The case with $n$ terms, that is 
\[
%\left(
\stackrel[i=1]{n}{*}\zeta_{x^{\left(i\right)},\,a^{\left(i\right)}}%\right)\left(N+1\right)
=\sum_{i=1}^{n}\zeta_{x^{\left(i\right)},\,a^{\left(i\right)}\prod_{1\leqslant k\ne i\leqslant n}\psi_{x^{\left(k\right)}}}%\left(N+1\right)
\]
is proved by induction on $n.$ Here we use the notations
\[
\zeta_{x^{\left(i\right)},\,a^{\left(i\right)}}\left(N\right)=\sum_{p=1}^{\infty}\frac{a_{p}^{\left(i\right)}}{\left(x_{p}^{\left(i\right)}\right)^{N}}.
\]
Assume this is true for $n-1$ and consider 
\begin{large}
\[
\stackrel[i=1]{n}{*}\zeta_{x^{\left(i\right)},a^{\left(i\right)}}
%\left(N+1\right)
=\zeta_{x^{\left(n\right)},\,a^{\left(n\right)}}\stackrel[i=1]{n-1}{*}\zeta_{x^{\left(i\right)},\,a^{\left(i\right)}}.
\]
\end{large}
Notice that by the induction hypothesis, this is (always computed at $N+1$ so this argument is omitted for simplicity) 
\begin{align*}
  \stackrel[i=1]{n}{*}\zeta_{x^{\left(i\right)},\,a^{\left(i\right)}} &=\zeta_{x^{\left(n\right)},\,a^{\left(n\right)}}*\sum_{i=1}^{n-1}\zeta_{x^{\left(i\right)},\,a^{\left(i\right)}\prod_{1\leqslant k\ne i\leqslant n-1}\psi_{x^{\left(k\right)}}} =\sum_{i=1}^{n-1}\zeta_{x^{\left(n\right)},\,a^{\left(n\right)}}*\zeta_{x^{\left(i\right)},\,a^{\left(i\right)}\prod_{k\ne i}\psi_{x^{\left(k\right)}}}.
\end{align*}
Using \eqref{eq:2terms}, each convolution term in the sum is evaluated as
\begin{align*}
\zeta_{x^{\left(n\right)},\,a^{\left(n\right)}}*\zeta_{x^{\left(i\right)},\,a^{\left(i\right)}\prod_{k\ne i}\psi_{x^{\left(k\right)}}} & =\zeta_{x^{\left(i\right)},\,a^{\left(i\right)}\prod_{1\leqslant k\ne i\leqslant n}\psi_{x^{\left(k\right)}}}
  +\zeta_{x^{\left(n\right)},\,a^{\left(n\right)}\psi_{a^{\left(i\right)}\prod_{1\leqslant k\ne i\leqslant n-1}\psi_{x^{\left(k\right)}}}}.
\end{align*}
Next, we show that
\begin{equation}
\sum_{i=1}^{n-1}\zeta_{x^{\left(n\right)},\,a^{\left(n\right)}\psi_{a^{\left(i\right)}\prod_{1\leqslant k\ne i\leqslant n-1}\psi_{x^{\left(k\right)}}}}=\zeta_{x^{\left(n\right)},\,a^{\left(n\right)}\prod_{k\ne i}\psi_{x^{\left(k\right)}}}\label{eq:zeta}
\end{equation}
by showing equivalently that 
\begin{equation}
\sum_{i=1}^{n-1}\psi_{x^{\left(i\right)},a^{\left(i\right)}\prod_{1\leqslant k\ne i\leqslant n-1}\psi_{x^{\left(k\right)}}}\left(z\right)=\prod_{i=1}^{n-1}\psi_{x^{\left(i\right)},a^{\left(i\right)}}\left(z\right).\label{eq:psi}
\end{equation}
Starting from the elementary partial fraction decomposition
\[
\sum_{i=1}^{n-1}\frac{z}{x_{i}-z}\prod_{k\ne i}\frac{x_{i}}{x_{k}-x_{i}}=\prod_{l=1}^{n-1}\frac{z}{x_{l}-z},
\]
%produces
%\[
%\sum_{i=1}^{n-1}\frac{z}{x_{n_{i}}^{\left(n_{i}\right)}-z}\prod_{k\ne i}\frac{x_{n_{i}}^{\left(i\right)}}{x_{n_{k}}^{\left(k\right)}-x_{n_{i}}^{\left(i\right)}}=\prod_{l=1}^{n-1}\frac{z}{x_{n_{l}}^{\left(l\right)}-z}.
%\]
substituting each $x_i$  with $x_{n_i}^{\left(i\right)}$, multiplying by $a_{n_{1}}^{\left(1\right)}a_{n_{2}}^{\left(2\right)}\ldots a_{n_{n-1}}^{\left(n-1\right)}$
and summing over the indices $n_{1},\dots,n_{n-1}$ 
produces
\begin{align*}
\sum_{n_{1},\dots,n_{n-1}}a_{n_{1}}^{\left(1\right)}a_{n_{2}}^{\left(2\right)}\ldots a_{n_{n-1}}^{\left(n-1\right)}\sum_{i=1}^{n-1}\frac{z}{x_{n_{i}}^{\left(i\right)}-z}\prod_{k\ne i}\frac{x_{n_{i}}^{\left(i\right)}}{x_{n_{k}}^{\left(k\right)}-x_{n_{i}}^{\left(i\right)}}
&=\sum_{n_{1},\dots,n_{n-1}}a_{n_{1}}^{\left(1\right)}a_{n_{2}}^{\left(2\right)}\ldots a_{n_{n-1}}^{\left(n-1\right)}\prod_{l=1}^{n-1}\frac{z}{x_{n_{l}}^{\left(l\right)}-z}\\
&=\sum_{i=1}^{n-1}\psi_{x^{\left(i\right)},a^{\left(i\right)}\prod_{k\ne i}\psi_{x^{\left(k\right)}}}\left(z\right).
\end{align*}
But the multiple sum on the left hand side is recognized as
\[
\sum_{n_{1},\dots,n_{n-1}}a_{n_{1}}^{\left(1\right)}a_{n_{2}}^{\left(2\right)}\ldots a_{n_{n-1}}^{\left(n-1\right)}\sum_{i=1}^{n-1}\frac{z}{x_{n_{i}}^{\left(n_{i}\right)}-z}\prod_{k\ne i}\frac{x_{n_{i}}^{\left(i\right)}}{x_{n_{k}}^{\left(k\right)}-x_{n_{i}}^{\left(i\right)}}=\prod_{i=1}^{n-1}\psi_{x^{\left(i\right)},a^{\left(i\right)}}\left(z\right)
\]
which proves identity (\ref{eq:psi}) and consequently identity (\ref{eq:zeta}).
Therefore, we deduce 
\begin{align*}
&\zeta_{x^{\left(n\right)},\,a^{\left(n\right)}}*\sum_{i=1}^{n-1}\zeta_{x^{\left(i\right)},\,a^{\left(i\right)}\prod_{k\ne i}\psi_{x^{\left(k\right)}}}
=\sum_{i=1}^{n-1}\zeta_{x^{\left(n\right)},\,a^{\left(n\right)}}*\zeta_{x^{\left(i\right)},\,a^{\left(i\right)}\prod_{k\ne i}\psi_{x^{\left(k\right)}}}\\
&=\sum_{i=1}^{n-1}\zeta_{x^{\left(i\right)},\,a^{\left(i\right)}\prod_{1\leqslant k\ne i\leqslant n}\psi_{x^{\left(k\right)}}}
+\sum_{i=1}^{n-1}\zeta_{x^{\left(n\right)},\,a^{\left(n\right)}\psi_{a^{\left(i\right)}\prod_{1\leqslant k\ne i\leqslant n-1}\psi_{x^{\left(k\right)}}}}\\
&=\sum_{i=1}^{n-1}\zeta_{x^{\left(i\right)},\,a^{\left(i\right)}\prod_{1\leqslant k\ne i\leqslant n}\psi_{x^{\left(k\right)}}}
+\zeta_{x^{\left(n\right)},\,a^{\left(n\right)}\prod_{1\leqslant k\ne i\leqslant n}\psi_{x^{\left(k\right)}}}\\
&=\sum_{i=1}^{n}\zeta_{x^{\left(i\right)},\,a^{\left(i\right)}\prod_{1\leqslant k\ne i\leqslant n}\psi_{x^{\left(k\right)}}}.
\end{align*}
which is the desired result. \QED
\subsection{Proof of Theorem \ref{thm:Generalized Ramanujan}}
With the choice 
\[
x_{n}=\dfrac{n^{2}\pi^{2}}{\beta},\,\,\,y_{n}=\dfrac{n^{2}\pi^{2}}{-\alpha},
\]
the associated generating functions are
\[
\psi_{x}\left(z\right)=\sum_{n=1}^{\infty}\frac{\beta z}{\pi^{2}n^{2}-\beta z}=\frac{1-\sqrt{\beta z}\,\cot\sqrt{\beta z}}{2},
\]
and
\[
\psi_{y}\left(z\right)=\sum_{n=1}^{\infty}\frac{-\alpha z}{\pi^{2}n^{2}+\alpha z}=\frac{1-\sqrt{\alpha z}\,\coth\sqrt{\alpha z}}{2},
\]
so that
\[
\psi_{x}\left(y_{n}\right)=\frac{1}{2}\left[1-n\pi\sqrt{-\dfrac{\beta}{\alpha}}\,\cot\left(n\pi\sqrt{-\dfrac{\beta}{\alpha}}\right)\right]=\frac{1}{2}\left[1-n\pi\sqrt{\dfrac{\beta}{\alpha}}\coth\left(n\pi\sqrt{\dfrac{\beta}{\alpha}}\right)\right]
\]
Similarly, we find that
\[
\psi_{y}\left(x_{n}\right)=\frac{1}{2}\left[1-n\pi\sqrt{\dfrac{\alpha}{\beta}}\,\coth\left(n\pi\sqrt{\dfrac{\alpha}{\beta}}\right)\right],
\]
and we obtain
\begin{align*}
&\dfrac{1}{\pi^{2N+2}}\sum_{k=1}^{N}\beta^{k}\left(-\alpha\right)^{N+1-k}\zeta\left(2k\right)\zeta\left(2N+2-2k\right) \\ &=\sum_{n=1}^{\infty}\dfrac{\left(-\alpha\right)^{N+1}}{\pi^{2N+2}\,n^{2N+2}}\left[\dfrac{1}{2}\left(1-n\pi\sqrt{\dfrac{\beta}{\alpha}}\,\coth\left(n\pi\sqrt{\dfrac{\beta}{\alpha}}\right)\right)\right]
  \\
  &+\sum_{n=1}^{\infty}\frac{\beta^{N+1}}{\pi^{2N+2}\,n^{2N+2}}\left[\frac{1}{2}\left(1-n\pi\sqrt{\dfrac{\alpha}{\beta}}\,\coth\left(n\pi\sqrt{\dfrac{\alpha}{\beta}}\right)\right)\right].
\end{align*}
After simplification and exchanging $\alpha\mapsto-\alpha$ and $\beta\mapsto-\beta,$ we find that
\begin{align*}
&\sum_{k=1}^{N}\left(-\beta\right)^{k}\alpha^{N+1-k}\,\zeta\left(2k\right)\zeta\left(2N+2-2k\right)  =\frac{\alpha^{N+1}+\left(-\beta\right)^{N+1}}{2}\,\zeta\left(2N+2\right)\\
 &-\pi\,\frac{\alpha^{N+1}}{2}\sum_{n=1}^{\infty}\frac{1}{n^{2N+1}}\,\sqrt{\frac{\beta}{\alpha}}\,\coth\left(\pi n\sqrt{\frac{\beta}{\alpha}}\right) -\pi\,\frac{\left(-\beta\right)^{N+1}}{2}\sum_{n=1}^{\infty}\frac{1}{n^{2N+1}}\,\sqrt{\frac{\alpha}{\beta}}\,\coth\left(\pi n\sqrt{\frac{\alpha}{\beta}}\right).
\end{align*}
Denoting $\mu=\beta/\alpha$, we can rewrite the above identity as follows
\begin{align*}
&\sum_{k=1}^{N}\left(-1\right)^{k}\mu^{k}\zeta\left(2k\right)\zeta\left(2N+2-2k\right)  =\frac{\mu^{N+1}+\left(-1\right)^{N+1}}{2}\,\zeta\left(2N+2\right)\\
 &-\frac{\pi}{2}\sum_{n=1}^{\infty}\frac{1}{n^{2N+1}}\sqrt{\mu}\,\coth\left(\pi n\sqrt{\mu}\right) -\frac{\pi\left(-\mu\right)^{N+1}}{2}\sum_{n=1}^{\infty}\frac{1}{n^{2N+1}}\sqrt{\frac{1}{\mu}}\,\coth\left(\dfrac{\pi n}{\sqrt{\mu}}\right).
\end{align*}
Simplifying further we finally obtain
\begin{align*}
&\sum_{k=1}^{N}\left(-1\right)^{k-1}\mu^{k}\zeta\left(2k\right)\zeta\left(2N+2-2k\right)  =-\frac{\mu^{N+1}+\left(-1\right)^{N+1}}{2}\,\zeta\left(2N+2\right)\\
 &+\frac{\pi\sqrt{\mu}}{2}\sum_{n=1}^{\infty}\frac{\coth\left(\pi n\sqrt{\mu}\right)}{n^{2N+1}} +\frac{\pi\left(-1\right)^{N+1}}{2}\,\mu^{N+\frac{1}{2}}\sum_{n=1}^{\infty}\frac{1}{n^{2N+1}}\,\coth\left(\frac{\pi n}{\sqrt{\mu}}\right).
\end{align*}
\subsection{Proof of Theorem \ref{thm:Bernoulli n terms}} %and identity (\ref{eq:Bernoulli 2 terms})}

We start with identity (\ref{eq:Bernoulli 2 terms}) by computing
\begin{align*}
\psi_{x}\left(z\right) & =\sum_{n\in \ZZ\setminus{\{0\}}}e^{2\imath\pi ny_{1}}\frac{z}{\frac{2\imath\pi n}{\omega_{1}}-z}=2\Re\sum_{n=1}^{\infty}e^{2\imath\pi ny_{1}}\frac{z}{\frac{2\imath\pi n}{\omega_{1}}-z} =2\Re\sum_{n=1}^{\infty}e^{2\imath\pi ny_{1}}\frac{z\left(-\frac{2\imath\pi n}{\omega_{1}}-z\right)}{\left(\dfrac{2\pi n}{\omega_{1}}\right)^{2}+z^{2}}\\
 & =-2\left\{\sum_{n=1}^{\infty}\dfrac{z^{2}}{\left(\dfrac{2\pi n}{\omega_{1}}\right)^{2}+z^{2}}\cos\left(2\pi ny_{1}\right)-\frac{z\left(2\pi n/\omega_1\right)}{\left(\dfrac{2\pi n}{\omega_{1}}\right)^{2}+z^{2}}\sin\left(2\pi ny_{1}\right)\right\}.
%\\&= -2\sum_{n=1}^{\infty}\dfrac{z^{2}}{\left(\dfrac{2\pi n}{\omega_{1}}\right)^{2}+z^{2}}\cos\left(2\pi ny_{1}\right)+2\sum_{n=1}^{\infty}\frac{z\left(2\pi n/\omega_1\right)}{\left(\dfrac{2\pi n}{\omega_{1}}\right)^{2}+z^{2}}\sin\left(2\pi ny_{1}\right)
\end{align*}
Both  sums are identified as Fourier series expansions: using \cite[(1.51) and (1.53)]{Oberhettinger}
produces, under the conditions $0 \leqslant y_1 \leqslant 1$ and $0 \leqslant y_2 \leqslant 1,$
\[
\frac{1}{z^{2}}+2\sum_{n=1}^{\infty}\frac{\cos\left(2\pi ny_{1}\right)}{z^{2}+\frac{4\pi^{2}}{\omega_{1}^{2}}n^{2}}=\frac{\omega_{1}}{2z}\frac{\cosh\left[\omega_{1}z\left(\frac{1}{2}-y_{1}\right)\right]}{\sinh\left(\dfrac{\omega_{1}z}{2}\right)}
\]
and
\[
\sum_{n=1}^{\infty}\frac{n}{z^{2}+\frac{4\pi^{2}}{\omega_{1}^{2}}n^{2}}\sin\left(2\pi ny_{1}\right)=\frac{\omega_{1}^{2}}{8\pi}\frac{\sinh\left[\omega_{1}z\left(\frac{1}{2}-y_{1}\right)\right]}{\sinh\left(\dfrac{\omega_{1}z}{2}\right)},
\]
so that we have
\begin{align*}
\psi_{x}\left(z\right) & =1-z\,\frac{\omega_{1}}{2}\frac{\cosh\left[\omega_{1}z\left(\frac{1}{2}-y_{1}\right)\right]}{\sinh\left(\dfrac{\omega_{1}z}{2}\right)}+z\,\frac{\omega_{1}}{2}\frac{\sinh\left[\omega_{1}z\left(\frac{1}{2}-y_{1}\right)\right]}{\sinh\left(\dfrac{\omega_{1}z}{2}\right)}\\
& =1-z\,\frac{\omega_{1}}{2}\frac{e^{\omega_{1}z\left(y_{1}-\frac{1}{2}\right)}}{\sinh\left(\dfrac{\omega_{1}z}{2}\right)}
 =1-\omega_{1}z\frac{e^{\frac{-\omega_{1}z}{2}}}{e^{\frac{\omega_{1}z}{2}}-e^{-\frac{\omega_{1}z}{2}}}e^{\omega_{1}zy_{1}}%\\
 =1-\omega_{1}z\,\frac{e^{\omega_{1}zy_{1}}}{e^{\omega_{1}z}-1}.
% \\
%& =1-\omega_{1}z\,\frac{e^{\omega_{1}zy_{1}}}{e^{\omega_{1}z}-1}.
\end{align*}
Therefore, we deduce that
\[
\psi_{x}\left(y_{n}\right)=1-2\imath\pi n\,\frac{\omega_{1}}{\omega_{2}}\frac{e^{2\imath\pi n\frac{\omega_{1}}{\omega_{2}}y_{1}}}{e^{2\imath\pi n\frac{\omega_{1}}{\omega_{2}}}-1}, \,\,\,
\psi_{y}\left(x_{n}\right)=1-2\imath\pi n\,\frac{\omega_{2}}{\omega_{1}}\frac{e^{2\imath\pi n\frac{\omega_{2}}{\omega_{1}}y_{2}}}{e^{2\imath\pi n\frac{\omega_{2}}{\omega_{1}}}-1},
\]
and 
\begin{align*}
\sum_{n\in\ZZ\setminus\{0\}}&\left(\frac{b_{n}}{y_{n}^{N+1}}\,\psi_{x}\left(y_{n}\right)+\frac{a_{n}}{x_{n}^{N+1}}\,\psi_{y}\left(x_{n}\right)\right) \\
&= \sum_{n\in\ZZ\setminus\{0\}}\left(\frac{b_{n}}{y_{n}^{N+1}}\,\psi_{x}\left(y_{n}\right)\right) +\sum_{n\in\ZZ\setminus\{0\}}\left(\frac{a_{n}}{x_{n}^{N+1}}\,\psi_{y}\left(x_{n}\right)\right)
\\
& = \sum_{n\in\ZZ\setminus\{0\}}\frac{e^{2\imath\pi ny_{2}}}{\left(\dfrac{2\imath\pi n}{\omega_{2}}\right)^{N+1}}\left(1-2\imath\pi n\,\frac{\omega_{2}}{\omega_{1}}\frac{e^{2\imath\pi n\frac{\omega_{2}}{\omega_{1}}y_{2}}}{e^{2\imath\pi n\frac{\omega_{2}}{\omega_{1}}}-1}\right)  \\&\quad+ \sum_{n\in\ZZ\setminus\{0\}}\frac{e^{2\imath\pi ny_{1}}}{\left(\dfrac{2\imath\pi n}{\omega_{1}}\right)^{N+1}}\left(1-2\imath\pi n\,\frac{\omega_{1}}{\omega_{2}}\frac{e^{2\imath\pi n\frac{\omega_{1}}{\omega_{2}}y_{1}}}{e^{2\imath\pi n\frac{\omega_{1}}{\omega_{2}}}-1}\right)  
\\
&= -\omega_{1}\sum_{n\in\ZZ\setminus\{0\}}\frac{e^{2\imath\pi ny_{2}}}{\left(\dfrac{2\imath\pi n}{\omega_{2}}\right)^{N}}\frac{e^{2\imath\pi n\frac{\omega_{1}}{\omega_{2}}y_{1}}}{e^{2\imath\pi n\frac{\omega_{1}}{\omega_{2}}}-1} -\omega_{2}^{N+1}\,\frac{\mathcal{B}_{N+1}\left(y_{2}\right)}{\left(N+1\right)!}\\ &\quad -\omega_{2}\sum_{n\in\ZZ\setminus\{0\}}\frac{e^{2\imath\pi ny_{1}}}{\left(\dfrac{2\imath\pi n}{\omega_{1}}\right)^{N}}\frac{e^{2\imath\pi n\frac{\omega_{2}}{\omega_{1}}y_{2}}}{e^{2\imath\pi n\frac{\omega_{2}}{\omega_{1}}}-1} -\omega_{1}^{N+1}\,\dfrac{\mathcal{B}_{N+1}\left(y_{2}\right)}{\left(N+1\right)!}.
\end{align*}
Then we finally have
\begin{align*}
\sum_{k=1}^{N}\omega_{1}^{k}\,\frac{\mathcal{B}_{k}\left(y_{1}\right)}{k!}\,\omega_{2}^{N+1-k}\,\frac{\mathcal{B}_{N+1-k}\left(y_{2}\right)}{\left(N+1-k\right)!} & =-\,\omega_{2}^{N+1}\,\frac{\mathcal{B}_{N+1}\left(y_{2}\right)}{\left(N+1\right)!}-\omega_{1}^{N+1}\,\frac{\mathcal{B}_{N+1}\left(y_{1}\right)}{\left(N+1\right)!}\\
 & -\omega_{1}\sum_{n\in\ZZ\setminus\{0\}}\frac{e^{2\imath\pi ny_{2}}}{\left(\dfrac{2\imath\pi n}{\omega_{2}}\right)^{N}}\frac{e^{2\imath\pi n\frac{\omega_{1}}{\omega_{2}}y_{1}}}{e^{2\imath\pi n\frac{\omega_{1}}{\omega_{2}}}-1} \\
 & -\omega_{2}\sum_{n\in\ZZ\setminus\{0\}}\frac{e^{2\imath\pi ny_{1}}}{\left(\dfrac{2\imath\pi n}{\omega_{1}}\right)^{N}}\frac{e^{2\imath\pi n\frac{\omega_{2}}{\omega_{1}}y_{2}}}{e^{2\imath\pi n\frac{\omega_{2}}{\omega_{1}}}-1}
\end{align*}
and 
\begin{align*}
\sum_{k=0}^{N+1}\omega_{1}^{k}\,\frac{\mathcal{B}_{k}\left(y_{1}\right)}{k!}\,\omega_{2}^{N+1-k}\,\frac{\mathcal{B}_{N+1-k}\left(y_{2}\right)}{\left(N+1-k\right)!}  =-\,\omega_{1}&\sum_{n\in\ZZ\setminus\{0\}}\frac{e^{2\imath\pi ny_{2}}}{\left(\dfrac{2\imath\pi n}{\omega_{2}}\right)^{N}}\frac{e^{2\imath\pi n\frac{\omega_{1}}{\omega_{2}}y_{1}}}{e^{2\imath\pi n\frac{\omega_{1}}{\omega_{2}}}-1}\\
  -\,\omega_{2}&\sum_{n\in\ZZ\setminus\{0\}}\frac{e^{2\imath\pi ny_{1}}}{\left(\dfrac{2\imath\pi n}{\omega_{1}}\right)^{N}}\frac{e^{2\imath\pi n\frac{\omega_{2}}{\omega_{1}}y_{2}}}{e^{2\imath\pi n\frac{\omega_{2}}{\omega_{1}}}-1}.
\end{align*}
The extension 
\[
\sum_{k_{1},k_2\ldots,k_{n}}\prod_{i=1}^{n}\omega_{i}^{k_{i}-1}\,\frac{\mathcal{B}_{k}\left(y_{i}\right)}{k_{i}!}=-\sum_{i=1}^{n}\frac{1}{\omega_{i}}\sum_{n\in\ZZ\setminus\{0\}}\frac{e^{2\imath\pi ny_{i}}}{\left(\dfrac{2\imath\pi n}{\omega_{i}}\right)^{N}}\prod_{j\ne i}\frac{e^{2\imath\pi n\frac{\omega_{i}}{\omega_{j}}y_{j}}}{e^{2\imath\pi n\frac{\omega_{i}}{\omega_{j}}}-1}
\]
based on the identity in Theorem \pageref{main_theorem} is straightforward and left as an exercise to
the reader.
\subsection{Proof of Corollaire \ref{cor:Euler}}
We use the  representation of Euler polynomials as integrals of Bernoulli polynomials
\[
\int_{x}^{x+\frac{1}{2}}\mathcal{B}_{n}\left(z\right)\mathrm{d}z=\frac{E_{n}\left(2x\right)}{2^{n+1}}.
\]
Starting from the identity
\[
\sum_{k_{1},k_2,\ldots,k_{n}}\prod_{i=1}^{n}\omega_{i}^{k_{i}-1}\,\frac{\mathcal{B}_{k_{i}}\left(y_{i}\right)}{k_{i}!}=-\sum_{i=1}^{n}\frac{1}{\omega_{i}}\sum_{m\in\ZZ\setminus\{0\}}\frac{e^{2\imath\pi my_{i}}}{\left(\dfrac{2\imath\pi m}{\omega_{i}}\right)^{N}}\prod_{j\ne i}\frac{e^{2\imath\pi m\frac{\omega_{j}}{\omega_{i}}y_{j}}}{e^{2\imath\pi m\frac{\omega_{j}}{\omega_{i}}}-1},
\]
let us integrate each over the variable $y_{i}$ from $x_{i}$ to $x_{i}+\frac{1}{2}.$
Thus the left-hand side is
\[
\sum_{k_{1},k_2,\ldots,k_{n}}\prod_{i=1}^{n}\omega_{i}^{k_{i}-1}\,\dfrac{E_{k_{i}}\left(2x_{i}\right)}{2^{2k_{i}+1}k_{i}!}=\frac{1}{2^{2N+2+n}}\sum_{k_{1},k_2,\ldots,k_{n}}\prod_{i=1}^{n}\omega_{i}^{k_{i}-1}\,\frac{E_{k_{i}}\left(2x_{i}\right)}{k_{i}!}.
\]
The right-hand side can now be computed using
\begin{equation*}
\int_{x_{j}}^{x_{j}+\frac{1}{2}}e^{2\imath\pi m\frac{\omega_{j}}{\omega_{i}}y_{j}}\,\mathrm{d}y_{j} 
%=\frac{1}{2\imath\pi m}\,\frac{\omega_{i}}{\omega_{j}}\left(e^{2\imath\pi m\frac{\omega_{j}}{\omega_{i}}\left(x_{j}+\frac{1}{2}\right)}-e^{2\imath\pi m\frac{\omega_{j}}{\omega_{i}}x_{j}}\right)
=\frac{e^{2\imath\pi m\frac{\omega_{j}}{\omega_{i}}x_{j}}}{2\imath\pi m}\,\frac{\omega_{i}}{\omega_{j}}\left(e^{i\pi m\frac{\omega_{j}}{\omega_{i}}}-1\right)
 %& =e^{2\imath\pi m\frac{\omega_{j}}{\omega_{i}}\left(x_{j}+\frac{1}{4}\right)}\frac{\sin\left(\pi\,\dfrac{m}{2}\,\dfrac{\omega_{j}}{\omega_{i}}\right)}{\pi m\frac{\omega_{j}}{\omega_{i}}}\\
 %& =\frac{1}{2}\,e^{2\imath\pi m\frac{\omega_{j}}{\omega_{i}}\left(x_{j}+\frac{1}{4}\right)}\,\text{sinc}\left(\pi\,\frac{m}{2}\,\frac{\omega_{j}}{\omega_{i}}\right)
\end{equation*}
and 
\[
\int_{x_{i}}^{x_{i}+\frac{1}{2}}e^{2\imath\pi my_{i}}\,\mathrm{d}y_{i}=\frac{e^{2\imath\pi mx_{i}}}{2\imath\pi m}\left(\left(-1\right)^{m}-1\right).
\]
We deduce the right-hand side after integration as
\begin{align*}
&\sum_{i=1}^{n}\frac{1}{\omega_{i}}\sum_{m\in\ZZ\setminus\{0\}}\frac{\frac{e^{2\imath\pi mx_{i}}}{2\imath\pi m}\left(\left(-1\right)^{m}-1\right)}{\left(\dfrac{2\imath\pi m}{\omega_{i}}\right)^{N}}\prod_{j\ne i}\frac{\frac{e^{2\imath\pi m\frac{\omega_{j}}{\omega_{i}}x_{j}}}{2\imath\pi m}\frac{\omega_{i}}{\omega_{j}}\left(e^{i\pi m\frac{\omega_{j}}{\omega_{i}}}-1\right)}{e^{2\imath\pi m\frac{\omega_{j}}{\omega_{i}}}-1}\\
&=\frac{-2}{\omega_{1}\omega_2\ldots\omega_{n}}\sum_{i=1}^{n}\sum_{m\,\,odd}\frac{\frac{e^{2\imath\pi mx_{i}}}{2\imath\pi m}}{\left(\dfrac{2\imath\pi m}{\omega_{i}}\right)^{N+n-1}}\prod_{j\ne i}\frac{e^{i2\pi m\frac{\omega_{j}}{\omega_{i}}x_{j}}}{e^{i\pi m\frac{\omega_{j}}{\omega_{i}}}+1},
\end{align*}
and finally after simplification we obtain\[
\sum_{k_{1},k_2\ldots,k_{n}}\prod_{i=1}^{n}\omega_{i}^{k_{i}-1}\,\frac{E_{k_{i}}\left(2x_{i}\right)}{k_{i}!}=\frac{2^{2N+n+3}}{\omega_{1}\omega_2\ldots\omega_{n}}\sum_{i=1}^{n}\sum_{m\,\,odd}\frac{\frac{e^{2\imath\pi mx_{i}}}{2\imath\pi m}}{\left(\dfrac{2\imath\pi m}{\omega_{i}}\right)^{N+n-1}}\prod_{j\ne i}\frac{e^{2\imath\pi m\frac{\omega_{j}}{\omega_{i}}x_{j}}}{e^{\pi \imath m\frac{\omega_{j}}{\omega_{i}}}+1}
\]
as desired. 

\subsection{Proof of Corollaire \ref{cor:zetasquare}}

The choice
\[
a_{p,q}=1,\,\,b_{r,s}=1,\,\,\text{and}\,\,
x_{p,q}=-\frac{p^{2}q^{2}}{\alpha^{2}},\,\,y_{r,s}=\frac{r^{2}s^{2}}{\beta^{2}}
\]
produces the Dirichlet series
\[
\zeta_{x,a}\left(s\right)=\left(-\alpha^{2}\right)^{s}\zeta^{2}\left(2s\right),\,\,\zeta_{y,b}\left(s\right)=\beta^{2s}\zeta^{2}\left(2s\right)
\]
and the generating functions
\[
\psi_{x}\left(z\right)=-\sum_{p=1}^{\infty}\sum_{q=1}^{\infty}\frac{\alpha^{2}z}{p^{2}q^{2}+\alpha^{2}z}=-\sum_{n=1}^{\infty}\tau_{0}(n)\,\frac{\alpha^{2}z}{n^{2}+\alpha^{2}z}.
\]
Here, $\tau_{0}(n)$ the number of divisors of $n$, and we have applied the general formula
\[\sum_{p=1}^{\infty}\sum_{q=1}^{\infty}F(pq)  =\sum_{n=1}^{\infty}\sum_{\underset{q|n}{q\geqslant1}}F(n)=\sum_{n=1}^{\infty}\tau_0(n)\,F(n).\]
Similarly we find that,
\[
\psi_{y}\left(z\right)=\sum_{m=1}^{\infty}\tau_0(m)\,\frac{\beta^{2}z}{m^{2}-\beta^{2}z}.
\]
Applying (\ref{eq:2terms}) produces
\begin{align*}
&\sum_{k=1}^{N}\left(-1\right)^{k}\alpha^{2k}\beta^{2N+2-2k}\zeta^{2}\left(2k\right)\zeta^{2}\left(2N+2-2k\right)\\
&=-\left(-\alpha^{2}\right)^{N+1}\sum_{p=1}^{\infty}\sum_{q=1}^{\infty}\frac{1}{\left(pq\right)^{2N+2}}\sum_{m=1}^{\infty}\frac{\tau_{0}(m)\beta^{2}p^{2}q^{2}}{\alpha^{2}m^{2}+\beta^{2}p^{2}q^{2}}-\beta^{2N+2}\sum_{r=1}^{\infty}\sum_{s=1}^{\infty}\frac{1}{\left(rs\right)^{2N+2}}\sum_{n=1}^{\infty}\frac{\tau_{0}(n)\alpha^{2}r^{2}s^{2}}{\beta^{2}n^{2}+\alpha^{2}r^{2}s^{2}}\\
&=-\left(-\alpha^{2}\right)^{N+1}\sum_{m=1}^{\infty}\sum_{n=1}^{\infty}\frac{\tau_{0}(n)\,\tau_{0}(m)}{n^{2N}}\,\frac{\beta^{2}}{\alpha^{2}m^{2}+\beta^{2}n^{2}}
-\left(\beta^{2}\right)^{N+1}\sum_{m=1}^{\infty}\sum_{n=1}^{\infty}\frac{\tau_{0}(n)\,\tau_{0}(m)}{m^{2N}}\,\frac{\alpha^{2}}{\beta^{2}n^{2}+\alpha^{2}m^{2}}.
\end{align*}
as desired. %This completes the proof of Corollaire \ref{cor:zetasquare}. 
\subsection{Proof of Corollaire \ref{cor:shifted zeta}}

With the choice
\[
a_{p,q}=q^{c},\,\,b_{r,s}=s^{d}\,\,\text{and}\,\,
x_{p,q}=-\frac{p^{2}q^{2}}{\alpha^{2}},\,\,y_{r,s}=\frac{r^{2}s^{2}}{\beta^{2}},
\]
we obtain the generating functions
\[
\psi_{x}\left(z\right)=\sum_{p=1}^{\infty}\sum_{q=1}^{\infty}\frac{q^{c}z}{-\left(p^{2}q^{2}/\alpha^2\right)-z}=-\sum_{m=1}^{\infty}\sum_{n=1}^{\infty}\frac{\alpha^{2}q^{c}z}{p^{2}q^{2}+z\alpha^{2}}=-\sum_{n=1}^{\infty}\tau_{c}(n)\,\frac{\alpha^{2}z}{n^{2}+\alpha^{2}z}
\]
and
\[
\psi_{y}\left(z\right)=\sum_{r=1}^{\infty}\sum_{s=1}^{\infty}\frac{s^{d}z}{\left(r^{2}s^{2}/\beta^2\right)-z}=\sum_{r=1}^{\infty}\sum_{s=1}^{\infty}\frac{\beta^{2}s^{d}z}{r^{2}s^{2}-z\beta^{2}}=\sum_{m=1}^{\infty}\tau_{d}(m)\,\frac{\beta^{2}z}{m^{2}-z\beta^{2}}
\]
associated with the Dirichlet functions 
\[
\zeta_{x,a}\left(k\right)=\left(-1\right)^{k}\alpha^{2k}\sum_{p=1}^{\infty}\sum_{q=1}^{\infty}\frac{q^{c}}{p^{2k}q^{2k}}=\left(-1\right)^{k}\alpha^{2k}\zeta\left(2k\right)\zeta\left(2k-c\right)
\]
and
\[
\zeta_{y,b}\left(k\right)=\beta^{2k}\sum_{r=1}^{\infty}\sum_{s=1}^{\infty}\frac{s^{d}}{r^{2k}s^{2k}}=\beta^{2k}\zeta\left(2k\right)\zeta\left(2k-d\right).
\]
This produces
\begin{align*}
\sum_{k=1}^{N}\left(-1\right)^{k}\alpha^{2k}\beta^{2N+2-2k} &\zeta\left(2k\right)\zeta\left(2k-c\right)\zeta\left(2N+2-2k\right)\zeta\left(2N+2-2k-d\right) \\ =&-\left(-\alpha^{2}\right)^{N+1}\sum_{m=1}^{\infty}\sum_{n=1}^{\infty}\tau_{c}(n)\,\tau_{d}(m)\,\frac{1}{n^{2N}}\,\frac{\beta^{2}}{\alpha^{2}m^{2}+\beta^{2}n^{2}}\\
&-\beta^{2N+2}\sum_{m=1}^{\infty}\sum_{n=1}^{\infty}\tau_{c}(n)\,\tau_{d}(m)\,\frac{1}{m^{2N}}\,\frac{\alpha^{2}}{\beta^{2} n^{2}+\alpha^{2}m^{2}}.
\end{align*}
\subsection{Proof of Theorem \ref{thm:The-Bessel-zeta}} We first compute the Bessel zeta generating function as follows:
\[\sum_{N=1}^{\infty} \zeta_{B,\nu}(N)z^N =-z\,\frac{\mathrm{d}}{\mathrm{d}z}\,\log j_{\nu}\left(\sqrt{z}\right)=\frac{\sqrt{z}}{2}\frac{J_{\nu+1}\left(\sqrt{z}\right)}{J_{\nu}\left(\sqrt{z}\right)} =\nu-\frac{\sqrt{z}}{2}\frac{J_{\nu-1}\left(\sqrt{z}\right)}{J_{\nu}\left(\sqrt{z}\right)},\]
where we used the linear recurrence
\[
J_{\nu+1}\left(\sqrt{z}\right)+J_{\nu-1}\left(\sqrt{z}\right)=\frac{2\nu}{\sqrt{z}}\,J_{\nu}\left(\sqrt{z}\right).
\]
%we deduce
%\[
%\psi\left(z\right)=\frac{\sqrt{z}}{2}\frac{J_{\nu-1}\left(\sqrt{z}\right)}{J_{\nu}\left(\sqrt{z}\right)}-\nu.
%\]
Now we make the choice
\[
x_{n}=\frac{j_{n,\nu}^{2}}{\beta},y_{n}=-\frac{j_{n,\nu}^{2}}{\alpha},
\]
which produces 
\[
\psi_{x}\left(z\right)=\nu-\frac{\sqrt{\beta z}}{2}\,\frac{J_{\nu-1}\left(\sqrt{\beta z}\right)}{J_{\nu}\left(\sqrt{\beta z}\right)},
\]
where the factor of $\sqrt{\beta}$ occurs since we consider $\zeta_{x,a}$ instead of $\zeta_{B,\nu}$. Using the fact that $J_{\nu}\left(\imath z\right)=\imath^{\nu}I_{\nu}\left(z\right)$ we get
\begin{align*}
\psi_{x}\left(y_{n}\right) & =\nu-\frac{1}{2}\,\imath\sqrt{\dfrac{\beta}{\alpha}}\,j_{n,\nu}\,\frac{J_{\nu-1}\left(\imath\sqrt{\dfrac{\beta}{\alpha}}\,j_{n,\nu}\right)}{J_{\nu}\left(\imath\sqrt{\dfrac{\beta}{\alpha}}\,j_{n,\nu}\right)} =\nu-\frac{1}{2}\,\sqrt{\dfrac{\beta}{\alpha}}\,j_{n,\nu}\,\frac{I_{\nu-1}\left(\sqrt{\dfrac{\beta}{\alpha}}\,j_{n,\nu}\right)}{I_{\nu}\left(\sqrt{\dfrac{\beta}{\alpha}}\,j_{n,\nu}\right)}
\end{align*}
while
%\[
%\psi_{y}\left(z\right)=\psi\left(-\alpha z\right)=\nu-\frac{\sqrt{-\alpha z}}{2}\,\frac{J_{\nu-1}\left(\sqrt{-\alpha z}\right)}{J_{\nu}\left(\sqrt{-\alpha z}\right)}
%\]
%and
\[
\psi_{y}\left(x_{n}\right)=\nu-\frac{1}{2}\,\imath\sqrt{\dfrac{\alpha}{\beta}}\,j_{n,\nu}\,\frac{J_{\nu-1}\left(\imath\sqrt{\dfrac{\alpha}{\beta}}\,j_{n,\nu}\right)}{J_{\nu}\left(\imath\sqrt{\dfrac{\alpha}{\beta}}\,j_{n,\nu}\right)}=\nu-\frac{1}{2}\sqrt{\dfrac{\alpha}{\beta}}\,j_{n,\nu}\,\frac{I_{\nu-1}\left(\sqrt{\dfrac{\alpha}{\beta}}\,j_{n,\nu}\right)}{I_{\nu}\left(\sqrt{\dfrac{\alpha}{\beta}}\,j_{n,\nu}\right)}.
\]
We deduce
\begin{align*}
\sum_{k=1}^{N}\alpha^{N-k+1}\left(-\beta\right)^{k}\zeta_{B,\nu}\left(k\right)\zeta_{B,\nu}\left(N+1-k\right) & =\sum_{q=1}^{\infty}\frac{\alpha^{N+1}}{j_{\nu,q}^{2N+2}}\left(\nu-\sqrt{\frac{\beta}{\alpha}}\frac{j_{\nu,q}}{2}\,\frac{I_{\nu-1}\left(\sqrt{\dfrac{\beta}{\alpha}}\,j_{\nu,q}\right)}{I_{\nu}\left(\sqrt{\dfrac{\beta}{\alpha}}\,j_{\nu,q}\right)}\right)\\
 +&\sum_{q=1}^{\infty}\frac{\left(-\beta\right)^{N+1}}{j_{\nu,q}^{2N+2}}\left(\nu-\sqrt{\frac{\alpha}{\beta}}\,\frac{j_{\nu,q}}{2}\frac{I_{\nu-1}\left(\sqrt{\dfrac{\alpha}{\beta}}\,j_{\nu,q}\right)}{I_{\nu}\left(\sqrt{\dfrac{\alpha}{\beta}}\,j_{\nu,q}\right)}\right).
\end{align*}
Finally after simplification we obtain
\begin{align*}
\sum_{k=1}^{N}\alpha^{N-k+1}\left(-\beta\right)^{k}\zeta_{B,\nu}\left(k\right)\zeta_{B,\nu}\left(N+1-k\right)  =\nu\left(\alpha^{N+1}+\left(-\beta\right)^{N+1}\right)\zeta_{B,\nu}\left(N+1\right) \\  -\frac{\alpha^{N+1}}{2}\sqrt{\frac{\beta}{\alpha}}\,\sum_{q=1}^{\infty}\frac{1}{j_{\nu,q}^{2N+1}}\,\frac{I_{\nu-1}\left(\sqrt{\dfrac{\beta}{\alpha}}\,j_{\nu,q}\right)}{I_{\nu}\left(\sqrt{\dfrac{\beta}{\alpha}}\,j_{\nu,q}\right)} -\frac{\left(-\beta\right)^{N+1}}{2}\sqrt{\frac{\alpha}{\beta}}\,\sum_{q=1}^{\infty}\frac{1}{j_{\nu,q}^{2N+1}}\,\frac{I_{\nu-1}\left(\sqrt{\dfrac{\alpha}{\beta}}\,j_{\nu,q}\right)}{I_{\nu}\left(\sqrt{\dfrac{\alpha}{\beta}}\,j_{\nu,q}\right)}
\end{align*}
as desired. %this completes the proof of Theorem \ref{thm:The-Bessel-zeta}.

\subsection{Proof of Theorem \ref{thm:Hurwitz}}
%\tw{This notation is super confusing: we have $\psi$ as digamma, $\psi$ as zeta generating function, $\psi_H$, and $\psi$ with either one or two arguments.}
Since we have the general identity
\[
\sum_{n=0}^{\infty}\frac{z}{\left(x+n\right)^{2}-z}=\frac{\sqrt{z}}{2}\left[\psi\left(x+\sqrt{z}\right)-\psi\left(x-\sqrt{z}\right)\right],
\]
with $\psi$ the digamma function, choosing
\[
x_{n}=\frac{\left(x+n\right)^{2}}{\beta},\,\,y_{n}=-\frac{\left(y+n\right)^{2}}{\alpha},
\]
produces the generating functions
%\[
%\psi_{x}\left(z\right)=\psi\left(x,\beta z\right),\,\,\psi_{y}\left(z\right)=\psi\left(y,-\alpha z\right)
%\]
%so that
\[
\psi_{x}\left(y_{n}\right)=\frac{1}{2}\,\imath\sqrt{\frac{\beta}{\alpha}}\left(y+n\right)\left[\psi\left(x+\imath\sqrt{\frac{\beta}{\alpha}}\left(y+n\right)\right)-\psi\left(x-\imath\sqrt{\frac{\beta}{\alpha}}\left(y+n\right)\right)\right]
\]
and
\[
\psi_{y}\left(x_{n}\right)=\frac{1}{2}\,\imath\sqrt{\frac{\alpha}{\beta}}\left(x+n\right)\left[\psi\left(y+\imath\sqrt{\frac{\alpha}{\beta}}\left(x+n\right)\right)-\psi\left(y-\imath\sqrt{\frac{\alpha}{\beta}}\left(x+n\right)\right)\right].
\]
Applying (\ref{eq:2terms}) and subtracting $1$ to both $x$ and
$y$ (as the summation index in the Hurwitz zeta function starts at
$0$) produces the following 
%\begin{align*}
%\\
%& =\beta^{N+1}\frac{i\sqrt{\frac{\alpha}{\beta}}}{2}\sum_{n=0}^{\infty}\frac{1}{\left(x+n\right)^{2N+1}}\left[\psi\left(y+\imath\sqrt{\frac{\alpha}{\beta}}\left(x+n\right)\right)-\psi\left(y-\imath\sqrt{\frac{\alpha}{\beta}}\left(x+n\right)\right)\right]%\\
%&+\left(-\alpha\right)^{N+1}\frac{i\sqrt{\frac{\beta}{\alpha}}}{2}\sum_{n=0}^{\infty}\frac{1}{\left(y+n\right)^{2N+1}}\left[\psi\left(x+\imath\sqrt{\frac{\beta}{\alpha}}\left(y+n\right)\right)-\psi\left(x-\imath\sqrt{\frac{\beta}{\alpha}}\left(y+n\right)\right)\right].
%\end{align*}
\begin{align*}
 &\sum_{k=1}^{N}\beta^{k}\left(-\alpha\right)^{N+1-k}\zeta_{H}\left(x;2k\right)\zeta_{H}\left(y;2N+2-2 k\right)\\
    & =\beta^{N+1}\frac{\imath}{2}\sqrt{\frac{\alpha}{\beta}}\sum_{n=0}^{\infty}\frac{1}{\left(x+n\right)^{2N+1}}\left[\psi\left(y+\imath\sqrt{\frac{\alpha}{\beta}}\left(x+n\right)\right)-\psi\left(y-\imath\sqrt{\frac{\alpha}{\beta}}\left(x+n\right)\right)\right]\\
    &+\left(-\alpha\right)^{N+1}\frac{\imath}{2}\sqrt{\frac{\beta}{\alpha}}\sum_{n=0}^{\infty}\frac{1}{\left(y+n\right)^{2N+1}}\left[\psi\left(x+\imath\sqrt{\frac{\beta}{\alpha}}\left(y+n\right)\right)-\psi\left(x-\imath\sqrt{\frac{\beta}{\alpha}}\left(y+n\right)\right)\right].
\end{align*}
\subsection{Proof of Theorem \ref{Dixit_thm}}
Let us prove first the following multisection identity:
\begin{equation}
\sum_{p=1}^{\infty}\frac{z^{2m}\,p^{m-1}}{z^{2m}-p^{2m}}=\frac{\pi z^{m}}{2m}\sum_{j=-\frac{m-1}{2}}^{\frac{m-1}{2}}\left(-1\right)^{j}\cot\left(\pi ze^{\imath\frac{j\pi}{m}}\right).\label{eq:multisection2}
\end{equation}
Using the Mittag-Leffler expansion of the cotangent function
\[\sum_{p=1}^{\infty}\frac{2z^{2}}{z^{2}-p^{2}}=-1+\pi z\cot\left(\pi z\right)\]
we deduce, with $m=2n+1,$ 
\begin{align*}
&\pi z\sum_{j=-\frac{m-1}{2}}^{\frac{m-1}{2}}\left(-1\right)^{j}\cot\left(\pi ze^{\imath\frac{j\pi}{m}}\right)  =\sum_{j=-n}^{n}\left(-1\right)^{j}\pi z\cot\left(\pi ze^{\imath\frac{j\pi}{m}}\right)\\
 & =\sum_{j=-n}^{n}\left(-1\right)^{j}e^{-\imath\frac{j\pi}{m}}\left(e^{\imath\frac{j\pi}{m}}\pi z\right)\cot\left(\pi ze^{\imath\frac{j\pi}{m}}\right)\\
 & =\sum_{j=-n}^{n}\left(-1\right)^{j}e^{-\imath\frac{j\pi}{m}}\left[1+\sum_{p=1}^{\infty}\frac{2\left(ze^{\imath\frac{j\pi}{m}}\right)^{2}}{\left(ze^{\imath\frac{j\pi}{m}}\right)^{2}-p^{2}}\right] \\
 &=\sum_{j=-n}^{n}\left(-1\right)^{j}e^{-\imath\frac{j\pi}{m}}+\sum_{p=1}^{\infty}\sum_{j=-n}^{n}\left(-1\right)^{j}e^{-\imath\frac{j\pi}{m}}\frac{2\left(ze^{\imath\frac{j\pi}{m}}\right)^{2}}{\left(ze^{\imath\frac{j\pi}{m}}\right)^{2}-p^{2}}.
\end{align*}
The first sum 
\[
\sum_{j=-n}^{n}\left(-1\right)^{j}e^{-\imath\frac{j\pi}{m}}=\sum_{j=-n}^{n}e^{\imath j\pi} e^{-\imath\frac{j\pi}{2n+1}} = \sum_{j=-n}^{n}e^{\imath j\pi\frac{2n}{2n+1}}=0
\]
since the summation range is over a residue system. Using a partial
fraction decomposition, the inner sum in the second term is 
\[
2z^{2}\sum_{j=-n}^{n}\frac{\left(-1\right)^{j}e^{\imath\frac{j\pi}{2n+1}}}{z^{2}\,e^{\imath\frac{2j\pi}{2n+1}}-p^{2}}=2z^{2}m\,\frac{p^{m-1}z^{m-1}}{z^{2m}-p^{2m}},
\]
so that finally we have
\[
\frac{\pi z^{m}}{2m}\sum_{j=-\frac{m-1}{2}}^{\frac{m-1}{2}}\left(-1\right)^{j}\cot\left(\pi ze^{\imath\frac{j\pi}{m}}\right)=\sum_{p = 1}^{\infty}\frac{p^{m-1}z^{2m}}{z^{2m}-p^{2m}}.
\]
Now let us check the main formula (\ref{eq:multisection identity}):
the generating functions associated with the choices (\ref{eq:choice1})
 are
\[
\psi_{x}\left(z\right)=\sum_{n=1}^{\infty}\frac{z}{\left(n^2/\alpha\right)-z}=\frac{1}{2}\left(1-\pi\sqrt{\alpha z}\,\cot\left(\pi\sqrt{\alpha z}\right)\right)
\]
and, by identity (\ref{eq:multisection}),
\begin{align*}
\psi_{y}\left(z^{2m}\right) & =\sum_{n=1}^{\infty}\frac{z^{2m}}{-\left(n^{2m}/\beta\right)-z^{2m}}=-\sum_{n=1}^{\infty}\frac{\beta z^{2m}}{z^{2m}+\beta z^{2m}} =-\sum_{n=1}^{\infty}\frac{\tilde{z}^{2m}}{n^{2m}+\beta\tilde{z}^{2m}}
\end{align*}
with
\[
\tilde{z}=z\beta^{\frac{1}{2m}}.
\]
Then,
\[
\psi_{y}\left(z^{2m}\right)=\frac{\pi z^{m}\sqrt{\beta}}{2m}\sum_{j=-\frac{m-1}{2}}^{\frac{m-1}{2}}\left(-1\right)^{j}\coth\left(\pi z\beta^{\frac{1}{2m}}e^{\imath\frac{j\pi}{m}}\right)
\]
and
\[
\psi_{y}\left(z\right)=\frac{\pi\sqrt{\beta z}}{2m}\sum_{j=-\frac{m-1}{2}}^{\frac{m-1}{2}}\left(-1\right)^{j}\coth\left(\pi\left(z\beta\right)^{\frac{1}{2m}}e^{\imath\frac{j\pi}{m}}\right).
\]
Therefore we deduce that
\[
\psi_{x}\left(y_{n}\right)=\frac{1}{2}\left(1-\pi\sqrt{\frac{\alpha}{\beta}}\,\,n^{m}\coth\left(\pi\sqrt{\frac{\alpha}{\beta}}n^{m}\right)\right)
\]
and
\[
\psi_{y}\left(x_{n}\right)=\frac{\pi n}{2m}\sqrt{\frac{\beta}{\alpha}}\sum_{j=-\frac{m-1}{2}}^{\frac{m-1}{2}}\left(-1\right)^{j}\coth\left(\pi n^{\frac{1}{m}}\left(\sqrt{\frac{\beta}{\alpha}}\right)^{\frac{1}{m}}e^{\imath\frac{j\pi}{m}}\right).
\]
The associated Dirichlet series are
\[
\zeta_{x}\left(k\right)=\alpha^{k}\zeta\left(2k\right),\,\,\zeta_{y}\left(k\right)=\left(-\beta\right)^{k}\sum_{n=1}^{\infty}\frac{n^{m-1}}{n^{2mk}}=\left(-\beta\right)^{k}\zeta\left(2mk-m+1\right).
\]
We obtain the convolution formula
\begin{align*}
&\sum_{k=1}^{N}\alpha^{k}\left(-\beta\right)^{N+1-k}\zeta\left(2k\right)\zeta\left(2m\left(N+1-k\right)-m+1\right) \\
&=\alpha^{N+1}\sum_{n=1}^{\infty}\frac{1}{n^{2N+2}}\,\frac{\pi n}{2m}\sqrt{\frac{\beta}{\alpha}}\sum_{j=-\frac{m-1}{2}}^{\frac{m-1}{2}}\left(-1\right)^{j}\coth\left(\pi n^{\frac{1}{m}}\left(\sqrt{\frac{\beta}{\alpha}}\right)^{\frac{1}{m}}e^{\imath\frac{j\pi}{m}}\right)\\
&+\left(-\beta\right)^{N+1}\sum_{n=1}^{\infty}\frac{n^{m-1}}{n^{2m\left(N+1\right)}}\,\frac{1}{2}\left[1-\pi\sqrt{\frac{\alpha}{\beta}}\,\,n^{m}\coth\left(\pi\sqrt{\frac{\alpha}{\beta}}\,n^{m}\right)\right]
\end{align*}
or, after simplification,
%\begin{align*}
%\end{align*}
\begin{align*}
 &\sum_{k=1}^{N}\alpha^{k}\left(-\beta\right)^{N+1-k}\zeta\left(2k\right)\zeta\left(2m\left(N-k\right)+m+1\right)\\ &=\alpha^{N+1}\,\dfrac{\pi}{2m}\,\sqrt{\frac{\beta}{\alpha}}\,\sum_{n=1}^{\infty}\frac{1}{n^{2N+1}}\sum_{j=-\frac{m-1}{2}}^{\frac{m-1}{2}}\left(-1\right)^{j}\coth\left(\pi n^{\frac{1}{m}}\left(\sqrt{\frac{\beta}{\alpha}}\right)^{\frac{1}{m}}e^{\imath\frac{j\pi}{m}}\right)\\
&+\frac{\left(-\beta\right)^{N+1}}{2}\zeta\left(2mN+m+1\right)-\left(-\beta\right)^{N+1}\,\frac{\pi}{2}\,\sqrt{\frac{\alpha}{\beta}}\,\sum_{n=1}^{\infty}\frac{1}{n^{2mN+1}}\coth\left(\pi\sqrt{\frac{\alpha}{\beta}}\,\,n^{m}\right)
\end{align*}
as desired. %This completes the proof. 
\subsection{Proof of Proposition \ref{Atul_Herglotz}} The choices
$$x_n = -\dfrac{n^2}{\beta},\,\, y_n = \dfrac{n^2}{\alpha},\,\,\,\textrm{and}\,\,\,  a_n = b_n = \dfrac{1}{n}$$
in our main setup produce
\[\zeta_{x,a}(k) = \left(-\beta\right)^k\zeta(2k+1) \,\,\,\text{and}\,\,\,\zeta_{y,b}(k) = \alpha^k\,\zeta(2k+1).\]
With $\psi(z)$ the digamma function, the associated generating functions are
\begin{align*}\psi_{x,a}(z) &= -\sum_{n=1}^{\infty}\dfrac{\beta z}{n\left(n^2+\beta z\right)}=-\dfrac{1}{2}\left[\psi\left(\imath\sqrt{\beta z} + 1\right) -\psi\left(-\imath\sqrt{\beta z} + 1\right) + 2\gamma\right]
\\
&=-\dfrac{1}{2}\left[\psi\left(\imath\sqrt{\beta z}\right) -\psi\left(-\imath\sqrt{\beta z}\right) + 2\gamma\right]
\end{align*}
and 
\begin{align*}\psi_{y,b}(z) &= \sum_{n=1}^{\infty}\dfrac{\alpha z}{n\left(n^2-\alpha z\right)}=-\dfrac{1}{2}\left[\psi\left(\imath\sqrt{\alpha z} + 1\right) -\psi\left(-\imath\sqrt{\alpha z} + 1\right) + 2\gamma\right]
\\
&=-\dfrac{1}{2}\left[\psi\left(\imath\sqrt{\alpha z}\right) -\psi\left(-\imath\sqrt{\alpha z}\right) + 2\gamma\right].
\end{align*}
Therefore we deduce
\[\psi_{x,a}\left(y_n\right) = -\dfrac{1}{2}\left[\psi\left(\imath n\sqrt{\dfrac{\beta}{\alpha}}\right)-\psi\left(-\imath n\sqrt{\dfrac{\beta}{\alpha}}\right)+2\gamma\right]\]
and
\[\psi_{y,b}\left
(x_n\right) = -\dfrac{1}{2}\left[\psi\left(\imath n\sqrt{\dfrac{\alpha}{\beta}}\right)-\psi\left(-\imath n\sqrt{\dfrac{\alpha}{\beta}}\right)+2\gamma\right].\]
We deduce, for all $\alpha, \beta \in \RR^+$
\begin{align*}
    &-\dfrac{1}{2}\sum_{n=1}^{\infty}\dfrac{\left(-\beta\right)^{N+1}}{n^{N+2}}\left[\psi\left(\imath n\sqrt{\dfrac{\beta}{\alpha}}\right)-\psi\left(-\imath n\sqrt{\dfrac{\beta}{\alpha}}\right)+2\gamma\right]
    \\
    &-\dfrac{1}{2}\sum_{n=1}^{\infty}\dfrac{\alpha^{N+1}}{n^{N+2}}\left[\psi\left(\imath n\sqrt{\dfrac{\beta}{\alpha}}\right)-\psi\left(-\imath n\sqrt{\dfrac{\beta}{\alpha}}\right)+2\gamma\right]
    \\
    &= \sum_{k=1}^{N}\zeta(2k+1)\,\zeta(2N-2k+3)\, \alpha^{N-k+1}\left(-\beta\right)^{k}.
\end{align*}
Simplifying the left hand--side using $\alpha\beta=4\pi^2$ produces the desired result. \QED

\section{Conclusion and Acknowledgements}

The main result of this paper, Theorem \ref{main_theorem}, follows from the direct application of the geometric sum formula: it simply expresses the $n-$fold convolution of Dirichlet series as the sum of $n$ Dirichlet series with modified weights.
One of its main advantages is that it explains why the $\alpha\beta=\pi^{2}$ condition is needed in these identities, which are essentially one-parameter
identities. It also reduces the proof of similar identities to the computation of the zeta generating function \eqref{gf}.

Several paths have not been explored yet and will be the subject of future work.  One of the difficulties associated with our result is the determination
of a closed form version for the generating function \eqref{gf} associated to a choice of Dirichlet series. A more thorough exploration of the known values of such generating functions would provide a better view on
this family of equivalent identities.

Due to their multivariate nature, these results can be extended to a choice of more exotic Dirichlet
series such as multiple zeta values or Witten zeta-functions. We explored a variant for double sums in Theorem \ref{thm:twoindicesgeneral}, although we barely scratch the surface.

%The asymptotic behavior of Dirichlet series is an important topic in number theory, and one may wonder what kind of information is provided by the present approach 

There are other relatively simple specializations of our main theorem, though for brevity we have not pursued these. For instance, let $j_\nu(z)$ denote a modified Bessel function of the first kind (as in \eqref{jnudef}), $j_{\nu,k}$ the $k$-th zero of $j_\nu(z)$ ordered by absolute magnitude, and 
$$\tilde{\zeta_{\nu}}(s)= \sum_{k=1}^{\infty} \frac{1}{j_{\nu+1}(j_{\nu,k})j_{\nu,k}^{s+2}}$$
the \textit{alternate Bessel zeta function}. Though the definition is somewhat unmotivated, it naturally appears when lifting identities for multiple zeta functions to identities for the zeros of Bessel zeta functions, with the Riemann zeta function being replaced by the alternate Bessel zeta function \cite{Christophe2}. Let $$P_n(z):= \sum_{m=0}^{n-1} \frac{\mathrm{d}^2}{\mathrm{d}z^{2n}} \left(\frac{1}{j_{\nu}(z)} \right) \Bigg\vert_{z=0}\frac{z^{2m}}{\left(2m\right)!}$$
denote the degree $2n-2$ polynomial obtained from the truncated Taylor series of $1/j_\nu(z)$ around $0$.
Using the arguments of \cite[Theorem 10]{Christophe2} shows that Krein's expansion is equivalent to 
$$ \frac{1}{j_\nu(z)} = P_n(z) + 4(\nu+1)\sum_{k=0}^{\infty} z^{2n+2k} \tilde{\zeta_{\nu}}(2n+2k).$$
Then, we can recognize Krein's expansion as a closed form for the zeta generating function \eqref{gf} attached to the alternate Bessel zeta function. This specialization of our main theorem produces an identity relating the zeros $j_{\nu,k}$ of the Bessel $j_\nu$ function with $j_{\nu+1}(j_{\nu,k})$. Variants of Krein's expansion also hold for other special function \cite{Krein}, and will likely give analogous results. Note that the zeros of many families of orthogonal polynomials, including Bessel functions, satisfy highly nontrivial sum relations \cite{Calogero}. Incorporating these known identities for Bessel zeros will likely simplify our results. We leave as an open problem the specialization of our identities to the zeros of other families of orthogonal polynomials. 
%We also note that identities for orthogonal polynomial zeros are closely tied to the quantum mechanical Thomas-Reiche-Kuhn (TRK) sum rules. Given a quantum potential $V(x)$, and solutions to the one-dimensional Schr\"odinger equation $\phi_n(x)$, the TRK rules gives identities for the zeros of $\phi_n(x)$. This is significant since Bessel functions are radial eigenfunctions for a free particle in spherical coordinates, while Airy functions arise as the solution to the ``quantum bouncer" \cite{Crandall} with confined potential $V(x)=\infty, x<0$ and $V(x) = kx, k> 0, x\geq 0$. Thus, the TRK rules give further nontrivial identities for Bessel and Airy zeros. The study of TRK sum rules will complement our existing work, as it presents a series of identities for the zeros of Schr\"odinger eigenfunctions which can be systematically integrated with our reciprocity relations. 

There has been other work on generalizing Ramanujan-type reciprocity to various arbitrary Dirichlet series. Under suitable convergence conditions which we omit, let
$$F\left(s\right):=\sum_{n=1}^{\infty} \frac{a_n}{\lambda^s_n} \,\,\,\text{and}\,\,\, G\left(s\right):=\sum_{n=1}^{\infty} \frac{b_n}{\mu^s_n}$$ denote two arbitrary Dirichlet series which satisfy the zeta--type functional equation 
$$\chi\left(s\right):=\left(2\pi\right)^s \,\Gamma\left(s\right)F\left(s\right) = (2\pi)^{s-\delta}\,\Gamma\left(\delta-s\right) G\left(\delta-s\right).$$ 
Then we have the reciprocity relation related to Lambert series
$$\sum_{n=1}^{\infty} a_n \,e^{-\lambda_n z} = \left(\frac{2\pi}{z} \right)^\delta\sum_{n=1}^{\infty}b_n\, e^{-4\pi^2\mu_n/z} + \frac{1}{2i\pi} \int_{\mathcal{C}}\dfrac{\left(2\pi\right)^t\chi\left(t\right)}{z^{t}}\,\mathrm{d}t,$$
where $\mathcal{C}$ is any curve enclosing the poles of $F(s)$ and $G(s)$,
along with other generalizations \cite{Berndt2021}. We wonder whether our main theorem can be applied to generate interesting quasimodular relations involving the $e^{-\lambda_n z}$ kernel, and whether other reciprocity relations for the Riemann zeta function lift to reciprocity relations for arbitrary Dirichlet series.

A. Dixit, R. Gupta, R. Kumar, and collaborators have studied many generalizations and analogues of Ramanujan's reciprocity formula \cite{Dixit, DixitSquare, Dixit2, Dixit3, Dixit4}. We have rederived some of their results in this paper. We invite the reader to use our main result to rigorously rederive some of their other reciprocity results, as this method may extend the parameter domains under which their identities hold. Our method may also extend to other quasimodular type transformations. For example, A. Dixit, R. Gupta, and R. Kumar study the higher Herglotz functions \cite{DixitHerglotz} and obtain several reciprocity relations for them. Another example is Ramanujan reciprocity over imaginary quadratic number fields \cite[Theorem 1.3]{Kumar1}, though this would require the development of a number field analog of the Koshliakov kernel \eqref{koshseries}.

\section{Acknowledgements}
Data sharing not applicable to this article as no datasets were generated or analysed during the current study.
The authors would like to thank Atul Dixit for his guidance and support
throughout the completion of this work and for taking time to read a rough draft of this manuscript. Christophe Vignat thanks
the staff at the Mathematics Library, Ecole Normale Sup\'{e}rieure, Paris for providing access to this remarkable place.


\begin{thebibliography}{10}
\bibitem{Lost Notebook}G. E. Andrews and B.C. Berndt, Ramanujan\textquoteright s
Lost Notebook, Part IV, Springer, New York, 2013. 

\bibitem{Kumar1}S. Banerjee and R. Kumar, Explicit identities on zeta values over imaginary quadratic field, 2021. %Pages: 1--25, 
%arXiv:2105.04141.

\bibitem{Berndt 2}B. C. Berndt, Ramanujan\textquoteright s formula
for $\zeta(2n+1)$, in Professor Srinivasa Ramanujan Commemoration
Volume, Jupiter Press, Madras, India, 2--9, 1974. 

\bibitem{Notebooks 2}B. C. Berndt, Ramanujan\textquoteright s
Notebooks, Part II, Springer, New York, 1989. 

\bibitem{Notebook 3}B. C. Berndt, Ramanujan\textquoteright s Notebooks,
Part III, Springer, New York, 1991. 
\bibitem{Berndt3}B. C. Berndt, An unpublished manuscript of Ramanujan
on infinite series identities, Journal of Ramanujan Mathematical Society,
19, 57--74, 2004.

\bibitem{Berndt2021}B. C. Berndt, A. Dixit, R. Gupta, and A. Zaharescu, A class of identities associated with Dirichlet series satisfying Hecke's functional equation, 2021. Pages: 1--15, arXiv:2108.13991.

\bibitem{Berndt}B. C. Berndt and A. Straub, Ramanujan's formula for
$\zeta(2n+1)$, in: Exploring the Riemann Zeta Function: 190 years
from Riemann's Birth, H. Montgomery, A. Nikeghbali and M. Th. Rassias
Eds, Springer, 13--34, 2017.

\bibitem{Calogero}F. Calogero, On the zeros of Bessel functions, Lettere al Nuovo Cimento, 20:7, 254-256, 1977.

\bibitem{Chavan} S. Chavan. An elementary proof of Ramanujan's formula for $\zeta(2m+1)$, 2021, preprint. 

\bibitem{Crandall}R. E. Crandall, On the quantum zeta function, Journal of Physics. A. Mathematical and General, 29:21, 6795--6816, 1996.


\bibitem{DixitSquare}A. Dixit and R. Gupta, On squares of odd zeta values and analogues of Eisenstein series, {Adv. in Appl. Math.}, 110, 86 -- 119, 2019.

\bibitem{DixitHerglotz}A. Dixit, R. Gupta, and R. Kumar, Extended higher Herglotz functions I. Functional equations, \url{https://arxiv.org/abs/2107.02607}, 2021.

\bibitem{Dixit3}A. Dixit, R. Gupta, R. Kumar, and B. Maji, {Generalized {L}ambert series, {R}aabe's cosine transform and a generalization of {R}amanujan's formula for {$\zeta(2m+1)$}}, {Nagoya Math. J.}, 239, 232 -- 293, 2020.

\bibitem{Dixit4}A. Dixit, A. Kesarwani, and R. Kumar, A generalized modified Bessel function and explicit transformations of certain Lambert series, \url{https://arxiv.org/abs/2012.12064}, 2020.


\bibitem{Dixit2}A. Dixit, A. Kesarwani, and V. H. Moll, {A generalized modified {B}essel function and a higher level analogue of the theta transformation formula}, {J. Math. Anal. Appl.}, 459:1, 385 -- 418, 2018.

\bibitem{Dixit}A. Dixit and B. Maji, Generalized Lambert series and
arithmetic nature of odd zeta values, Proceedings of the Royal Society
of Edinburgh Section A: Mathematics, 150:2, 741 -- 769, 2020.


\bibitem{Gun}S. Gun, M. Ram Murty and P. Rath, Transcendental values
of certain Eichler integrals, Bulletin of the London Mathematical
Society, 1--14, 2011. 


\bibitem{Katayama}K. Katayama, On Ramanujan's formula for values
of Riemann zeta-function at positive odd integers, Acta Arithmetica
22, 149--155, 1973.


\bibitem{Komori}Y. Komori, K. Matsumoto and H. Tsumura, Barnes multiple
zeta-functions, Ramanujan's formula, and relevant series involving
hyperbolic functions, J. Ramanujan Math. Soc., 28--1, 49--69, 2013.

\bibitem{Lerch}M. Lerch, Sur la fonction $\zeta(s)$ pour les valeurs
impaires de l\textquoteright argument, J. Sci. Math. Astron, %pub.
%pelo Dr. F. Gomes Teixeira, 
Coimbra 14, 65--69, 1901. 



\bibitem{SL} S.L. Malurkar, On the application of Herr Mellin’s integrals to some series, J. Indian
Math. Soc. 16 (1925/26), 130–138

\bibitem{Nanjundiah}T.S. Nanjundiah, Certain summations due to Ramanujan,
and their generalisations, Proc. Indian Academy of Science, Sect.
A 34, 215--228, 1951. 

\bibitem{Oberhettinger}F. Oberhettinger, Fourier Expansions, a collection
of formulas, Academic Press, 1973.

\bibitem{Ramanujan}S. Ramanujan, Notebooks, Tata Institute of Fundamental
Research, Bombay, 1957, 2nd edition, 2012. 

\bibitem{Romik}D. Romik and R. Scherer, Alternative summation orders for the {E}isenstein series {$G_2$} and {W}eierstrass {$\wp$}-function, {Rocky Mountain J. Math.}, 50:4, 1473--1482, 2020. 

\bibitem{Krein}V. B. Sherstyukov and E. V. Sumin, Reciprocal expansion of modified Bessel function in simple fractions and obtaining general summation relationships containing its zeros, Journal of Physics: Conf. Series 937, 1--5, 2017,

\bibitem{Terras}A. Terras, Some formulas for the Riemann zeta function
at odd integer argument resulting from Fourier expansions of the Epstein
zeta function, Acta Arithmetica, 29, 1976. 

\bibitem{Christophe2}T. Wakhare and C. Vignat, Multiple zeta values for classical special functions, Ramanujan J., 51:3, 519--551, 2020. 

\bibitem{Christophe1}T. Wakhare and C. Vignat, Structural properties of multiple zeta values, Int. J. Number Theory, 17:8, 1873--1897, 2021. 

\end{thebibliography}
\end{document}